\documentclass[12pt]{amsart}

\usepackage{amsmath}
\usepackage{amssymb}
\usepackage{amsthm}
\usepackage[mathscr]{eucal}
\usepackage{a4wide,fullpage}
\usepackage[T1]{fontenc}
\usepackage{enumerate}
\usepackage[foot]{amsaddr}
\usepackage{thmtools}
\usepackage[pagebackref]{hyperref}
\usepackage{cleveref}

\newtheorem{theorem}{Theorem}[section]
\newtheorem{proposition}[theorem]{Proposition}
\newtheorem{lemma}[theorem]{Lemma}
\newtheorem{corollary}[theorem]{Corollary}

\newtheorem{remark}[theorem]{Remark}
\theoremstyle{definition}
\newtheorem{definition}[theorem]{Definition}

\numberwithin{equation}{section}

\newcommand{\C}{\mathbb{C}}
\newcommand{\R}{\mathbb{R}}
\newcommand{\Q}{\mathbb{Q}}
\newcommand{\Z}{\mathbb{Z}}

\newcommand{\F}{\mathbb{F}}
\renewcommand{\H}{\mathbb{H}}
\newcommand{\calA}{\mathcal{A}}
\newcommand{\calD}{\mathcal{D}}

\newcommand{\calF}{\mathcal{F}}

\newcommand{\calH}{\mathcal{H}}
\newcommand{\calM}{\mathcal{M}}
\newcommand{\calO}{\mathcal{O}}

\newcommand{\calQ}{\mathcal{Q}}
\newcommand{\calS}{\mathcal{S}}

\newcommand{\Leg}{\mathrm{Leg}}
\newcommand{\ord}{\mathrm{ord}}
\newcommand{\leg}[2]{\left( \frac{#1}{#2} \right)}

\newcommand{\kzxz}[4]{\left(\begin{smallmatrix} #1 & #2 \\ #3 & #4\end{smallmatrix}\right) }
\newcommand{\kabcd}{\kzxz{a}{b}{c}{d}}
\newcommand{\dif}{\mathrm{d}}
\newcommand{\ST}{\mathrm{ST}}
\newcommand{\pihol}{\pi_{\mathrm{hol}}}

\let\Im\relax
\DeclareMathOperator{\Im}{Im}
\DeclareMathOperator{\Sl}{SL}

\DeclareMathOperator{\Aut}{Aut}
\DeclareMathOperator{\End}{End}
\DeclareMathOperator{\lcm}{lcm}
\DeclareMathOperator{\tr}{tr}
\DeclareMathOperator{\Fr}{Fr}
\DeclareMathOperator{\rank}{rank}
\DeclareMathOperator{\Tr}{Tr}

\title{Negative bias in moments of the Legendre family of elliptic curves}
\author{Ben Kane}
\email{bkane@hku.hk}
\author{Mikul\'{a}\v{s} Zindulka}
\email{zinmik2@gmail.com}
\address{Department of Mathematics, University of Hong Kong, Pokfulam, Hong Kong}
\date{\today}

\begin{document}

\begin{abstract}
We determine the bias in higher moments of the Legendre family of elliptic curves in view of the Negative Bias Conjecture of S. J. Miller. We show that every lower order term is either zero or negative on average for both even and odd moments, and explicitly compute the third and fourth moment. The results about Hurwitz class numbers in arithmetic progressions which we prove in the process may be useful for other applications.
\end{abstract}

\maketitle

\section{Introduction and statement of results}
\label{sec:intro}

There are various important statistics connected with the distribution of traces of Frobenius of elliptic curves. The \emph{Sato-Tate conjecture}, proved in \cite{BGHT, CHT}, predicts that for a fixed elliptic curve $ E/\Q $ without CM, the normalized traces of the reductions of $ E $ over $ \F_p $ are equidistributed with respect to the Sato-Tate measure as $ p \to \infty $. Inspired by the conjecture, Birch \cite{Bi} proved that when we simultaneously consider all (isomorphism classes of) elliptic curves over $ \F_p $, their normalized traces are also equidistributed with respect to the Sato-Tate measure.

Schoof \cite{Sch} resolved the closely related question how many elliptic curves $ E/\F_{p^r} $ have trace equal to some fixed $ t \in \Z $. It is in some sense more natural to consider the weighted count
\[
	N(p^r;t) := \sum_{\substack{E/\F_{p^r}\\\tr(E)=t}}\frac{1}{\left|\Aut_{\F_{p^r}}(E)\right|}.
\]
An elliptic curve $ E/\F_{p^r} $ is called \emph{supersingular} if $ p \mid \tr(E) $, and \emph{ordinary} otherwise. In the case when $ p \nmid t$, we have
\[
	2N(p^r;t) = H(4p^r-t^2),
\]
where $ H(4p^r-t^2) $ is the \emph{Hurwitz class number}.

One possible refinement is to determine the distribution for elliptic curves with trace $ \tr(E) \equiv m \pmod{M} $ for some $ m \in \Z $ and $ M \in \Z_{\geq 1} $. Bringmann, Pujahari, and the first author \cite[Theorem 1.1]{BKP} proved that also in this case the traces are equidistributed with respect to the Sato-Tate measure. Another possible refinement is to look at \emph{families of elliptic curves}. By a family $ \calF $, we mean a collection of elliptic curves
\[
	E_\lambda: y^2+a_1(\lambda)xy+a_3(\lambda)y = x^3+a_2(\lambda)x^2+a_4(\lambda)x+a_6(\lambda),
\]
where $ a_i(\lambda) \in \Z[\lambda] $. Let us assume that the discriminant $ \Delta(E_\lambda) $ is not the zero polynomial. If $ E_\lambda $ for some $ \lambda \in \Z $ is an elliptic curve over $ \Q $, then we can consider it as an elliptic curve over $ \F_p $ for all primes $ p $ of good reduction. More generally, we can let $ \lambda $ run over $ \F_{p^r} $ and view $ E_\lambda $ as an elliptic curve over the finite field (for some $ \lambda \in \F_{p^r} $, $ E_\lambda $ may be singular, but we still define its trace as $ \tr(E_\lambda) := p^r+1-\left|E(p^r)\right| $). The basic questions are again how many elliptic curves in the family have a fixed trace $ t $ and what is the distribution of traces in the family as $ p \to \infty $.

Families of elliptic curves were considered from a different point of view by S. J. Miller and his collaborators \cite{As,Mil1,Mil2}. Their motivation came from Random Matrix Theory (RMT), namely the Katz-Sarnak philosophy \cite{KS1,KS2}. In this context, a family is a collection of geometric objects sharing some properties together with their $ L $-functions. Vaguely speaking, Katz and Sarnak propose that zeros of the $ L $-functions attached to the objects in the family should behave like eigenvalues of random matrices from a particular group. Miller studied the so-called \emph{$ 1 $-} and \emph{$ 2 $-level densities} of zeros of $ L $-functions for various families in \cite{Mil1,Mil2} and observed that they match the predictions of Katz and Sarnak.

The \emph{$ k $-th moment} of a family $ \calF $ is defined by
\[
	S_k^{\calF}(p^r) := \sum_{\lambda \in \F_{p^r}}\tr(E_\lambda)^k
\]
These moments were considered mostly over $ \F_p $. For the first moment, Nagao \cite{Na} conjectured that
\[
	\lim_{X \to \infty}\frac{1}{X}\sum_{p \leq X}-S_1^{\calF}(p)\frac{\log p}{p} = \rank \calF(\Q(\lambda)).
\]
The formula was proved by Rosen and Silverman \cite{RS} in the case when $ \calF $ is a \emph{rational surface}.

Michel \cite{Mich} proved that if the $ j $-invariant $ j(E_\lambda) $ is non-constant, then the second moment satisfies
\[
	S_2^{\calF}(p) = p^2+O\left(p^{3/2}\right).
\]
The lower-order terms in the expansion of the second moment were studied by Miller \cite{Mil1,Mil2}, who observed that in all the examples considered, the first term after the leading term which does not average to zero has negative average. The \emph{Negative Bias Conjecture} states that this is true for every family with non-constant $ j(E_\lambda) $. Further examples supporting the conjecture can be found in \cite{As}.

Kazalicki and Naskr\k{e}cki \cite{KN} considered families
\[
	E_\lambda: y^2 = P(x)+\lambda Q(x),
\]
where $ P(x), Q(x) \in \Z[x] $ are polynomials of degrees $ \deg(P), \deg(Q) \leq 3 $. A family of this form is called a \emph{pencil of cubics}. They proved the Negative Bias Conjecture in this setting under additional constraints \cite[Theorem 1.4]{KN}, and gave a precise formulation of the conjecture. Namely, they pointed out that the $ k $-th moment is \emph{``motivic''} \cite[p. 5]{KN}. A consequence of this fact is that it can be decomposed as \cite[(1.1)]{KN}
\[
	S_k^{\calF}(p^r) = \sum_{j \in \frac{1}{2}\Z}\sum_{i}(\alpha_{k,j,i}(p))^r
\]
for some algebraic numbers $ \alpha_{k,j,i}(p) $ such that $ |\alpha_{k,j,i}(p)| = p^j $. We let
\[
	S_{k,j}^{\calF}(p^r) := \frac{1}{p^{rj}}\sum_{i}(\alpha_{k,j,i}(p))^r
\]
and define the averages
\[
	\calA_{k,r,j}^{\calF} := \lim_{x \to \infty}\frac{1}{\pi(x)}\sum_{p \leq x}S_{k,j}^{\calF}(p^r).
\]
The leading term in the second moment is of size $ p^{2r} $. The Negative Bias Conjecture can now be reformulated more precisely as follows: if $ j_0 $ is the largest $ j < 2 $ such that $ \calA_{2,r,j}^{\calF} \neq 0 $, then $ \calA_{2,r,j_0}^{\calF} $ is negative.

Recently, Pujahari, Yang, and the first author studied the bias in the second moment for elliptic curves with a trace $ \tr(E) \equiv m \pmod{M} $ and proved the surprising result that it can be both positive and negative depending on the congruence class \cite[Theorems 1.1 and 1.2]{KPY}.

One example where the distribution of traces is well understood is the \emph{Legendre family} $ \Leg $ given by
\[
	E_{\lambda}^{\Leg}: y^2 = x(x-1)(x-\lambda).
\]
The first two moments were computed by Miller \cite[Section 12.1]{Mil1}:
\begin{align*}
	S_1^{\Leg}(p^r)& = 0,\\
	S_2^{\Leg}(p^r)& = p^{2r}-2p^r-1.
\end{align*}
To be precise, Miller considered them only over $ \F_p $ but extending the result over $ \F_{p^r} $ does not present a difficulty. Ono, Saad, and Saikia investigated the family in connection with the \emph{${_2F_1}$-Gaussian hypergeometric functions} and found an expression for $ S_k^{\Leg}(p^r) $ in terms of the Hurwitz class numbers \cite[Proposition 2.11]{OSS} (the proposition is formulated in terms of moments of $ {_2F_1} $ but the two formulations are equivalent). In \cite[Theorem 1.1]{OSS}, they obtained the main term of the asymptotics for $ S_k^{\Leg}(p^r) $, and as an immediate consequence \cite[Corollary 1.2]{OSS} showed that the normalized traces in the Legendre family have the Sato-Tate distribution. In fact, we prove a slightly more precise version of \cite[Proposition 2.11]{OSS} as \Cref{prop:OSS}, the difference being that the contribution from supersingular curves is evaluated explicitly.

In this paper, we study the bias in the higher moments of the Legendre family. Our main result is an exact formula for the moments depending on the coefficients of certain cusp forms. It is stated in terms of the numbers defined for $ 0 \leq \mu \leq \left\lfloor\frac{k}{2}\right\rfloor $ by
\begin{equation}
\label{eq:T}
	T(k,\mu) := \frac{k-2\mu+1}{k-\mu+1}\binom{k}{\mu},
\end{equation}
which were introduced in \cite[Theorem 2.3]{KPY}. In particular,
\[
	C_k = T(2k,k) = \frac{1}{k+1}\binom{2k}{k}
\]
is the $ k $-th \emph{Catalan number}.

\begin{theorem}
\label{thm:Main}
Let $ p \geq 5 $ be a prime, $ r \in \Z_{\geq 1} $, and $ k \in \Z_{\geq 1} $. If $ k $ is even, then
\[
	S_k^{\Leg}(p^r) = C_{\frac{k}{2}}\left(p^{r(\frac{k}{2}+1)}-2p^{\frac{rk}{2}}\right)-3\sum_{\mu=1}^{\frac{k}{2}-1}T(k,\mu)p^{r\mu}-1+E_k(p^r)
\]
and if $ k $ is odd, then
\[
	S_k^{\Leg}(p^r) = \begin{cases}
	-2\sum_{\mu=1}^{(k-1)/2} T(k,\mu)p^{r\mu}+E_k(p^r) &\text{if $ p^r \equiv 1 \pmod{4} $,}\\
	0&\text{if $ p^r \equiv 3 \pmod{4} $,}
	\end{cases}
\]
where $ E_k(p^r) $ is a certain explicit expression depending on the coefficients of cusp forms.
\end{theorem}
\begin{proof}
For $ k $ even this follows from \Cref{thm:Seven}. For $ k $ odd this follows from \Cref{thm:Sodd} if $ p^r \equiv 1 \pmod{4} $ and from \Cref{prop:OSS} if $ p^r \equiv 3 \pmod{4} $.
\end{proof}

To calculate the averages $ \calA_{k,r,j}^{\Leg} $, we make use of an observation previously made in \cite{KPY}. Suppose that $ f $ is a newform of weight $ \kappa $ on a congruence subgroup $ \Gamma $. By a conjecture of Weil proved by Deligne \cite{Del}, the Fourier coefficients of $ f $ satisfy
\[
	c_f(p) = \beta_p+\gamma_p,
\]
where $ |\beta_p| = |\gamma_p| = p^{\frac{\kappa-1}{2}} $ are the complex roots of the Hecke polynomial. Similarly, the coefficients $ c_f(p^r) $ can be expressed as homogeneous polynomials of degree $ r $ in $ \beta_p $ and $ \gamma_p $. For the even moments, the coefficients come from newforms of even weight, hence they contribute to terms of size $ p^{rj} $ for $ j \in \Z+\frac{1}{2} $. For the odd moments, they come from newforms of odd weight, and hence contribute to terms of size $ p^{rj} $ for $ j \in \Z $. However, the contribution of these terms to the averages $ \calA_{k,r,j}^{\Leg} $ is always zero because the normalized coefficients are equidistributed (by the Sato-Tate conjecture for Fourier coefficients of cusp forms \cite{CHT}). The terms coming from $ E_k(p^r) $ can therefore be completely disregarded, and we have the following (a more precise version including the values of the averages $ \calA_{k,r,j}^{\Leg} $ will be given as \Cref{cor:Bias'}).

\begin{corollary}
\label{cor:Bias}
Let $ r \in \Z_{\geq 1} $ and $ k \in \Z_{\geq 1} $. All lower order terms in the asymptotic expansion of the $ k $-th moment $ S_k^{\Leg}(p^r) $, where $ p $ is a prime, are either negative or zero on average.
\end{corollary}

When $ k $ is fixed, we can evaluate $ E_k(p^r) $ explicitly. We do this for the third and fourth moment in \Cref{prop:S3,prop:S4}. In contrast, the elementary approach does not suffice for the computation of higher moments.

The rest of the paper is organized as follows. In \Cref{sec:prelim}, we collect the preliminaries. In \Cref{sec:tr}, we find a formula for the number of curves in the Legendre family with a fixed trace $ t $ and use it to express the moments in terms of the Hurwitz class numbers (\Cref{prop:OSS}). In \Cref{sec:hol}, we apply holomorphic projection to the moments of Hurwitz class numbers, and we find formulas for these moments in \Cref{sec:hur}. Finally, we use them to express the moments in the Legendre family and prove the main theorem in \Cref{sec:S}, where we also compute $ S_k^{\Leg}(p^r) $ for $ k \in \{3,4\} $.

\section*{Acknowledgments}
The research of the first author was supported by grants from the Research Grants Council of the Hong Kong SAR, China (project numbers HKU 17314122, HKU 17305923). This project was partially carried out while the second author was a PhD student at Charles University, Faculty of Mathematics and Physics, where he was supported by Charles University project PRIMUS/24/SCI/010.

\section{Preliminaries}
\label{sec:prelim}

\subsection{Holomorphic and non-holomorphic modular forms}

In this section, we give an overview of the theory of modular forms and other modular objects. More information can be found in \cite{Ko,On1}.

We use the usual notation $ \tau = u+iv \in \H $ for an element of the upper half-plane, and $ q := e^{2\pi i\tau} $. For $ d \in \Z $ odd, let
\[
	\varepsilon_d := \begin{cases}
		1&\text{if $ d \equiv 1 \pmod{4} $,}\\
		i&\text{if $ d \equiv 3 \pmod{4} $.}
	\end{cases}
\]
Let $ F: \H \to \C $ be a function on the upper half-plane. For $ \kappa \in \frac{1}{2}\Z $, the \emph{weight $ \kappa $ slash operator} is defined as follows. If $ \kappa \in \Z $, then
\[
	(F\vert_\kappa \gamma)(\tau) := (c\tau+d)^{-\kappa}F(\gamma\tau)
\]
for $ \gamma = \kabcd \in \Sl_2(\Z) $, and if $ \kappa \in \frac{1}{2}+\Z $, then
\[
	(F\vert_\kappa \gamma)(\tau) := \leg{c}{d}^{2\kappa}\varepsilon_d^{2\kappa}(c\tau+d)^{-\kappa}F(\gamma\tau)
\]
for $ \gamma = \kabcd \in \Gamma_0(4) $, where $ \leg{\cdot}{\cdot} $ denotes the \emph{extended Legendre symbol}.

Let $ \Gamma \subset \Sl_2(\Z) $ be a congruence subgroup containing $ T := \kzxz{1}{1}{0}{1} $ such that $ \Gamma \subset \Gamma_0(4) $ if $ \kappa \in \frac{1}{2}+\Z $. The function $ F $ satisfies modularity of weight $ \kappa \in \frac{1}{2}\Z $ on $ \Gamma $ with character $ \chi $ if
\[
	F\vert_\kappa \gamma = \chi(d)F
\]
for every $ \gamma \in \Gamma $. If $ F $ is also holomorphic on $ \H $ and $ F(\tau) $ grows at most polynomially in $ v $ as $ \tau = u+iv \to \Q\cup\{i\infty\} $, then $ F $ is called a \emph{(holomorphic) modular form}. A holomorphic modular form which vanishes at the cusps of $ \Gamma $ is called a \emph{cusp form}. The weight $ \kappa $ modular forms, respectively cusp forms, on $ \Gamma $ with character $ \chi $ form a complex vector space which we denote by $ M_\kappa(\Gamma,\chi) $, respectively $ S_\kappa(\Gamma,\chi) $. To determine whether two modular forms are equal, we use the Sturm bound \cite[Corollary~9.20]{Ste}, a consequence of the \emph{valence formula}.

%%%%%%%%%%%%%%%%%%%%%%%%%%%%%%%%%%%%%%%%%%%%%%%%%%%%%%%%%%%%%%%%%%%%%%%%
%\begin{lemma}[Sturm bound]
%\label{lem:sturm}
%Let $ \kappa \in \Z_{\geq 2} $, $ \Gamma \in \Sl_2(\Z) $ be a congruence subgroup, and $ \chi $ be a Dirichlet character. If $ f = \sum_{n=0}^\infty c_f(n)q^n $ and $ g = \sum_{n=0}^\infty c_g(n)q^n $ are two modular forms in $ M_\kappa(\Gamma,\chi) $ such that
%\[
%	c_f(n) = c_g(n)\qquad\text{for}\qquad 0 \leq n \leq \left\lfloor \frac{\kappa m}{12} \right\rfloor,
%\]
%where $ m := [\Sl_2(\Z):\Gamma] $, then $ f = g $.
%\end{lemma}
%%%%%%%%%%%%%%%%%%%%%%%%%%%%%%%%%%%%%%%%%%%%%%%%%%%%%%%%%%%%%%%%%%%%%%%%

There are also different types of non-holomorphic modular forms. For $ \kappa \in \frac{1}{2}\Z $, the \emph{weight $ \kappa $ hyperbolic Laplace operator} is defined by
\[
	\Delta_\kappa := -v^2\left(\frac{\partial^2}{\partial u^2}+\frac{\partial^2}{\partial v^2}\right)+i\kappa v\left(\frac{\partial}{\partial u}+i\frac{\partial}{\partial v}\right).
\]

A function $ F: \H \to \C $ is called a \emph{harmonic Maass form of weight $ \kappa \in \frac{1}{2}\Z $} if $ F $ is smooth, $ F $ satisfies weight $ \kappa $ modularity, $ \Delta_\kappa(F) = 0 $, and there exists $ a \in \R $ such that
\[
	F(u+iv) = O\left(e^{av}\right)\text{ as $ v \to \infty $}\quad\text{and}\quad F(u+iv) = O\left(e^{\frac{a}{v}}\right)\text{ for $ u \in \Q $ as $ v \to 0^+ $}.
\]
A harmonic Maass form which is also holomorphic on $ \H $ is called a \emph{weakly holomorphic modular form}. The \emph{$ \xi $-operator of weight $ \kappa $} is defined by $ \xi_\kappa := 2iv^{\kappa}\overline{\frac{\partial}{\partial\overline{\tau}}} $. If $ F $ is a harmonic Maass form of weight $ \kappa $, then $ \xi_\kappa F $ is a weakly holomorphic modular form of weight $ 2-\kappa $ called the \emph{shadow} of $ F $.

A function $ F: \H \to \C $ satisfying weight $ \kappa $ modularity is an \emph{almost holomorphic modular form} if there exist holomorphic functions $ F_j: \H \to \C $ for $ 0 \leq j \leq \ell $ such that $ F(\tau) = \sum_{j=0}^\ell F_j(\tau)v^{-j} $. The function $ F_0 $ is called a \emph{quasimodular form}.

Next, we review the facts about some modularity-preserving operators. Let $ \delta \in \Z_{\geq 1} $. If $ F: \H \to \C $, then the \emph{$ V $-operator} is defined by
\[
	F\vert V_\delta(\tau) := F(\delta\tau).
\]
If $ F $ has the Fourier expansion $ F(\tau) = \sum_{n \in \Z}c_{F,v}(n)q^n $, then the \emph{$ U $-operator} is defined by
\[
	F\vert U_\delta(\tau) := \sum_{n \in \Z}c_{F,\frac{v}{\delta}}(\delta n)q^n.
\]

For $ N_0, N_1 \in \Z_{\geq 1} $, let $ \Gamma_{N_0, N_1} := \Gamma_0(N_0) \cap \Gamma_1(N_1) $. Suppose that $ F $ satisfies weight $ \kappa \in \frac{1}{2}\Z $ modularity on $ \Gamma_{N_0, N_1} $ with character $ \chi $, where $ N_1 \mid N_0 $ and $ 4 \mid N_0 $ if $ \kappa \in \frac{1}{2}+\Z $. The function $ F \vert U_\delta $ satisfies weight $ \kappa $ modularity on $ \Gamma_{\lcm(N_0,\delta),N_1} $ with character $ \chi\cdot\leg{\delta}{\cdot}^{2\kappa} $, and $ F \vert V_\delta $ satisfies weight $ \kappa $ modularity on $ \Gamma_{N_0\delta, N_1} $ with character $ \chi\cdot\leg{\delta}{\cdot}^{2\kappa} $.

Let $ M \in \Z_{\geq 1} $ and $ m \in \Z $. If $ F = \sum_{n \in \Z}c_{F,v}(n)q^n $, then the \emph{sieving operator} is defined by
\[
	F\vert S_{M,m}(\tau) := \sum_{\substack{n\in\Z\\n \equiv m \pmod{M}}}c_{F,v}(n)q^n.
\]

The operator acts on functions satisfying weight $ \kappa $ modularity as follows.

\begin{lemma}
\label{lem:S}
Let $ M \in \Z_{\geq 1} $ and $ m \in \Z $. Let $ \kappa \in \frac{1}{2}\Z $, $ N_0, N_1 \in \Z_{\geq 1} $ such that $ N_1 \mid N_0 $ and $ 4 \mid N_0 $ if $ \kappa \in \frac{1}{2}+\Z $. Let $ \chi $ be a Dirichlet character of conductor $ N_{\chi} $. Let
\[
	\Gamma := \begin{cases}
		\Gamma_{\lcm(N_0,M^2,MN_1,MN_{\chi}), \lcm(N_1,M)}&\text{if $ \kappa \in \Z $ or $ M \not\equiv 2 \pmod{4} $,}\\
		\Gamma_{\lcm(N_0,4M^2,MN_1,MN_{\chi}), \lcm(N_1, M)}&\text{if $ \kappa \in \frac{1}{2}+\Z $ and $ M \equiv 2 \pmod{4} $.}
	\end{cases}
\]
If $ F: \H \to \C $ satisfies weight $ \kappa $ modularity on $ \Gamma_{N_0, N_1} $ with character $ \chi $, then $ F\vert S_{M,m} $ satisfies weight $ \kappa $ modularity on $ \Gamma $ with character $ \chi $.

Moreover, if $ M \mid 24 $, then $ \Gamma $ can be replaced with
\[
	\Gamma' := \begin{cases}
		\Gamma_{\lcm(N_0,M^2,MN_1,MN_{\chi}), N_1}&\text{if $ \kappa \in \Z $ or $ M \not\equiv 2 \pmod{4} $,}\\
		\Gamma_{\lcm(N_0,4M^2,MN_1,MN_{\chi}), N_1}&\text{if $ \kappa \in \frac{1}{2}+\Z $ and $ M \equiv 2 \pmod{4} $.}
	\end{cases}
\]
\end{lemma}
\begin{proof}
The case $ \kappa \in \Z $ is perhaps folklore and the case $ \kappa \in \frac{1}{2}+\Z $ is \cite[Lemma 2.3(2)]{BK2}.
\end{proof}

The \emph{$ \psi $-twist} of $ F = \sum_{n \in \Z}c_{F, v}(n)q^n $ by a Dirichlet character $ \psi $ is defined by
\[
	(F\otimes\psi)(\tau) := \sum_{n \in \Z}\psi(n)c_{F,v}(n)q^n.
\]
If $ N_{\psi} $ is the conductor of $ \psi $ and $ F \in M_\kappa(\Gamma_0(N),\chi) $, then $ F\otimes\psi \in M_\kappa(\Gamma_0(\lcm(N,N_{\psi}^2)),\chi\psi^2) $.

For $ k \in \{0,1\} $ and a Dirichlet character $ \chi $ of conductor $ N_{\chi} $ such that $\chi(-1)=(-1)^k$, define a \emph{theta function} by
\[
	\theta_{k,\chi}(\tau) := \sum_{n \in \Z}\chi(n)n^kq^{n^2}.
\]
We have $ \theta_{0,\chi}(\tau) \in M_{\frac{1}{2}}\left(\Gamma_0(4N_{\chi}^2),\chi\right) $ if $ \chi $ is even, and $ \theta_{1,\chi}(\tau) \in S_{\frac{3}{2}}\left(\Gamma_0(4N_{\chi}^2),\chi\chi_{-4}\right) $ if $ \chi $ is odd, where $ \chi_{-4} $ is the non-principal character modulo $ 4 $ \cite[Theorem 1.44]{On1}. In particular, if $ \chi = \chi_0 $ is the trivial character, then we use the notation $ \theta_0 := \theta_{0,\chi_0} $.

For $ k \in \Z_{\geq 0} $, $ m \in \Z $, and $ M \in \Z_{\geq 1} $, let
\[
	\theta_{k,m,M}(\tau) := \sum_{\substack{n \in \Z\\n \equiv m \pmod{M}}}n^kq^{n^2}.
\]
If $ k = 0 $, then we write simply $ \theta_{m,M} $ for $ \theta_{k,m,M} $. We have $ \theta_{m,M} \in M_{\frac{1}{2}}\left(\Gamma_{4M^2,M}\right) $ and $ \theta_{1,m,M} \in S_{\frac{3}{2}}\left(\Gamma_{4M^2,M}\right) $ (this follows from \Cref{lem:S}).

\subsection{Hurwitz class numbers}
\label{subsec:h}

Let $ D \in \Z $ be a negative discriminant and let $ \calQ_D $ be the set of binary integral quadratic forms of discriminant $ D $. For $ Q \in \calQ_D $, let $ \Gamma_Q $ be the \emph{stabilizer group} of $ Q $ in $ \Sl_2(\Z) $. The \emph{$|D|$-th Hurwitz class number} is defined by
\[
	H(|D|) := \sum_{Q \in \calQ_D/\Sl_2(\Z)}\frac{1}{\omega_Q},
\]
where
\[
	\omega_Q := \frac{\left|\Gamma_Q\right|}{2} = \begin{cases}
		2&\text{if $ Q = a(x^2+y^2) $,}\\
		3&\text{if $ Q = a(x^2+xy+y^2) $,}\\
		1&\text{otherwise.}
	\end{cases}
\]
We set $ H(0) := -\frac{1}{12} $ and $ H(n) := 0 $ if $ n \equiv 1,2 \pmod{4} $ or $ n \notin \Z_{\geq 0} $.

For a negative discriminant $ d \in \Z $, the \emph{class number} $ h(d) $ is the number of classes of primitive binary integral quadratic forms of discriminant $ d $. Let $ h^*(d) $ denote the class number weighted by half of the automorphism group, i.e.,
\[
	h^*(d) := \begin{cases}
		\frac{h(d)}{3}&\text{if $ d = -3 $,}\\
		\frac{h(d)}{2}&\text{if $ d = -4 $,}\\
		h(d)&\text{else.}
	\end{cases}
\]
It follows from the definition that
\[
	H(|D|) = \sum_{\substack{f^2 \mid D\\\frac{D}{f^2}\equiv 0, 1 \pmod{4}}}h^*\left(\frac{D}{f^2}\right)
\]
for $ D \in \Z $ negative. In particular, if $ p \equiv 3 \pmod{4} $ is a prime, then $ H(p) = h^*(-p) $.

Consider the generating function
\[
	\calH(\tau) := \sum_{n=0}^\infty H(n)q^n.
\]

It is a remarkable fact due to Zagier that $ \calH $ can be completed to a harmonic Maass form.

\begin{theorem}[{\cite[Chapter 2, Theorem 2]{HZ}}]
The function
\[
	\widehat{\calH}(\tau) := \calH(\tau)+\frac{1}{8\pi\sqrt{v}}+\frac{1}{4\sqrt{\pi}}\sum_{n=1}^\infty n\Gamma\left(-\frac{1}{2},4\pi n^2v\right) q^{-n^2}
\]
is a harmonic Maass form of weight $ \frac{3}{2} $ on $ \Gamma_0(4) $ which grows at most polynomially towards all cusps.
\end{theorem}

It is useful to have a relation between $ H(p^2|D|) $ and $ H(|D|) $ for a prime $ p $. If $ D \in \Z_{<0} $ and $ p $ is odd, then we set $ D_p := \frac{D}{p^{2\alpha}} $, where $ \alpha \in \Z_{\geq 0} $ is such that $ 2\alpha \leq \ord_p(D) \leq 2\alpha+1 $.

\begin{lemma}[{\cite[Lemma 2.5]{KP1}}]
\label{lem:delta1}
If $ D \in \Z_{<0} $ and $ p $ is an odd prime, then
\[
	H\left(|D|p^2\right) = pH\left(|D|\right)+\left(1+\leg{D_p}{p}\right)H\left(|D_p|\right).
\]
\end{lemma}
\begin{proof}
This was stated as \cite[Lemma 2.5]{KP1} under the assumption that $ D $ is a discriminant. If $ D $ is not a discriminant, i.e., $ D \equiv 2, 3 \pmod{4} $, then neither are $ Dp^2 $ and $ D_p $, hence both sides are zero.
\end{proof}

We also need a version of the previous result which works for every prime $ p $, including $ p = 2 $.

\begin{lemma}
\label{lem:delta2}
Let $ D = \Delta f^2 $, where $ \Delta < 0 $ is a fundamental discriminant and $ f \in \Z_{\geq 1} $, and let $ p $ be a prime. If $ f = p^\alpha g $ with $ p \nmid g $, then
\[
	H(|D|p^2) = pH(|D|)+\left(1-\leg{\Delta}{p}\right)H(|\Delta|g^2).
\]
\end{lemma}
\begin{proof}
This formula can be found in the proof of \cite[Lemma 2.5]{KP1}. The assumption that $ p $ is odd in the statement of the lemma is not necessary (it is needed only in the last step).
\end{proof}

For $ k \in \Z_{\geq 0} $, $ m \in \Z $, and $ M \in \Z_{\geq 1} $, the \emph{$ k $-th moment of the Hurwitz class numbers} is defined by
\[
	H_{k,m,M}(n) := \sum_{t \equiv m \pmod{M}}H(4n-t^2)t^k.
\]
The simpler notation $ H_{m,M}(n) := H_{0,m,M}(n) $ is used for the sum. The even moments were treated in \cite[Theorem 1.1]{KP1} and the odd moments in \cite[Theorem 1.4]{BKP}.

\subsection{Elliptic curves and trace of Frobenius}

Elliptic curves are curves of genus $ 1 $ having a specified base point $ O $. An elliptic curve over a field $ K $ has a \emph{Weierstrass equation}
\[
	y^2+a_1xy+a_3y = x^3+a_2x^2+a_4x+a_6.
\]
If we assume that $ \mathrm{char}(K) \neq 2, 3 $, then $ E $ has a Weierstrass equation of the form
\[
	y^2 = x^3+Ax+B.
\]
The \emph{discriminant} and \emph{$ j $-invariant} of $ E $ are defined by
\begin{align*}
	\Delta(E)& := -16(4A^3+27B^2),\\
	j(E)& := -1728\frac{(4A)^3}{\Delta(E)}.
\end{align*}

Let $ E_1 $ and $ E_2 $ be two elliptic curves. A morphism $ \phi: E_1 \to E_2 $ is called an \emph{isogeny} if $ \phi(O) = O $ \cite[III.4]{Si}. The curves $ E_1 $ and $ E_2 $ are called \emph{isogenous} if there exists an isogeny $ \phi: E_1 \to E_2 $ such that $ \phi(E_1) \neq \{O\} $. Isogeny is an equivalence relation on elliptic curves over $ K $.

Let $ p $ be a prime and $ r \in \Z_{\geq 1} $. The \emph{trace of Frobenius} of $ E $ is defined by
\[
	\tr(E) = \tr_{\F_{p^r}}(E) := p^r+1-\left|E(\F_{p^r})\right|,
\]
where $ \left|E(\F_{p^r})\right| $ is the number of points on $ E $ (including the point at infinity). The \emph{Frobenius endomorphism} of an elliptic curve $ E/\F_{p^r} $ is the map $ \Fr: E \to E $ defined for a point $ P := (x, y) \in E $ by $ \Fr(P) := \left(x^{p^r}, y^{p^r}\right) $. The trace of $ \Fr $ is $ \tr(E) $. By Tate's isogeny theorem \cite[Theorem 1]{Ta}, two elliptic curves $ E_1 $ and $ E_2 $ over $ \F_{p^r} $ are isogenous if and only if $ \tr(E_1) = \tr(E_2) $.

The \emph{Hasse bound} is $ |\tr(E)| \leq 2p^{\frac{r}{2}} $. Thus, if we let
\[
	x(E) := \frac{\tr(E)}{2p^{\frac{r}{2}}}
\]
be the normalized trace of Frobenius, then $ x(E) \in [-1,1] $. For an elliptic curve $ E $ over $ \Q $, we may consider $ x(E) $ for every reduction of $ E $ to the finite field $ \F_p $. The distribution of $ x(E) $ is predicted by the \emph{Sato--Tate conjecture}, proved in \cite{BGHT, CHT, HST}: If $ E $ does not have complex multiplication, then for $ -1 \leq a \leq b \leq 1 $,
\[
	\lim_{N \to \infty}\frac{\left|\{p \leq N:\; a \leq x(E) \leq b\}\right|}{\left|\{p \leq N\}\right|} = \frac{2}{\pi}\int_a^b \sqrt{1-x^2}\dif x.
\]
The interpretation is that the normalized traces are equidistributed with respect to the \emph{Sato--Tate measure}, defined for a Lebesgue measurable set $ A \subset \R $ by
\[
	\mu_{\ST}(A) := \int_A f_{\ST}(x)\dif x,
\]
where
\[
	f_{\ST}(x) := \begin{cases}
		\frac{2}{\pi}\sqrt{1-x^2}&\text{if }x \in [-1, 1],\\
		0&\text{if }x \notin [-1,1].
	\end{cases}
\]

\section{Number of curves with a given trace in the Legendre family}
\label{sec:tr}

The \emph{Legendre family} $ \Leg $ is a family of elliptic curves $ E^{\Leg}_\lambda/\Q $ given for $ \lambda \in \Z\setminus\{0,1\} $ by
\[
	E^{\Leg}_\lambda: y^2 = x(x-1)(x-\lambda).
\]
The discriminant of $ E^{\Leg}_\lambda $ equals
\[
	\Delta(E^{\Leg}_\lambda) = 16\lambda^2(\lambda-1)^2
\]
and the $ j $-invariant equals
\[
	j(E^{\Leg}_\lambda) = \frac{16^2(\lambda^2-\lambda+1)}{\lambda^2(\lambda-1)^2}.
\]

Let $ p \geq 3 $ be a prime and $ r \in \Z_{\geq 1} $. By considering the curves in $ \Leg $ over $ \F_{p^r} $, we obtain the family $ E^{\Leg}_\lambda/\F_{p^r} $ for $ \lambda \in \F_{p^r}\setminus\{0,1\} $. Questions about the distribution of traces of Frobenius in this family were studied by several authors \cite{AT,Ka2,OSS}. In general, we are interested in the number of curves in the family with trace $ t $, defined by
\[
	N^{\Leg}(p^r;t) := \left|\left\{\lambda \in \F_{p^r},\lambda \neq 0,1:\; \tr\left(E^{\Leg}_{\lambda}\right)=t\right\}\right|.
\]
Auer and Top \cite{AT} showed when an elliptic curve $ E/\F_{p^r} $ is isogenous to a curve in the Legendre family and studied the supersingular curves. Katz \cite{Ka2} studied the ratios
\[
	\frac{N^{\Leg}(p^r;t)}{N^{\Leg}(p^r;-t)}.
\]
Ono, Saad, and Saikia \cite{OSS} gave what is the most complete treatment of the traces in $ \Leg $ thus far.

The main goal of this section is to prove a formula for the $ k $-th moment $ S_k^{\Leg}(p^r) $ stated in \Cref{prop:OSS}. It is already contained in the work of Ono, Saad, and Saikia \cite[Proposition 2.11]{OSS}, except that to make it completely explicit, we also evaluate the contribution coming from the supersingular curves.

Let us introduce one useful piece of notation. If $ r $ is even, then let $ \varepsilon(p^r) \in \{-1,1\} $ be such that
\[
	\varepsilon(p^r)p^{\frac{r}{2}} \equiv 1 \pmod{4}.
\]
In other words, $ \varepsilon(p^r) = \leg{-1}{p^{\frac{r}{2}}} $.

\subsection{Supersingular curves}

For a prime $ p \geq 3 $ and $ m := \frac{p-1}{2} $, let $ H_p(t) $ be the \emph{Deuring polynomial} defined by
\[
	H_p(t) := \sum_{i=0}^m \binom{m}{i}^2t^i.
\]
If $ \lambda \in \F_{p^r}\setminus\{0, 1\} $, then $ E^{\Leg}_{\lambda}/\F_{p^r} $ is supersingular if and only if $ H_p(\lambda) = 0 $ \cite[Theorem V.4.1 (b)]{Si}.

\begin{proposition}
\label{prop:SS}
Let $ p \geq 3 $ be a prime, $ r \in \Z_{\geq 1} $, and $ t \in \Z $. If $ p \mid t $, then
\[
	N^{\Leg}(p^r;t) = \begin{cases}
		3H(p)&\text{if $ r $ is odd, $ p^r \equiv 3 \pmod{4} $, $ t=0 $,}\\
		\frac{1}{2}(p-1)&\text{if $ r $ is even, $ t = \varepsilon(p^r) 2p^{r/2} $,}\\
		0&\text{otherwise}.
	\end{cases}
\]
\end{proposition}
\begin{proof}
The only possible trace for a supersingular curve in the Legendre family over $ \F_{p^r} $ is $ 0 $ if $ r $ is odd and $ \varepsilon(p^r)2p^{r/2} $ if $ r $ is even \cite[Lemma 2.1]{Ka2}.

First, let $ r $ be odd. If $ \lambda \in \overline{\F_p} $ is such that $ E^{\Leg}_\lambda $ is supersingular, then $ \lambda \in \F_{p^2} $ \cite[Proposition 2.2]{AT}. Hence, if $ \lambda \in \F_{p^r} $ and $ E_\lambda $ is supersingular, then $ \lambda \in \F_p $. By \cite[Proposition 3.2]{AT}, we have
\[
	N^{\Leg}(p;0) = \begin{cases}
		1&\text{if }p = 3,\\
		3h(-p)&\text{if }p \equiv 3 \pmod{4}, p > 3,\\
		0&\text{if }p \equiv 1 \pmod{4}.
	\end{cases}
\]
The result follows from $ H(3) = \frac{1}{3} $ and $ H(p) = h(-p) $ if $ p \equiv 3 \pmod{4} $, $ p > 3 $.

Secondly, let $ r $ be even. The polynomial $ H_p(t) $ has distinct roots in $ \overline{\F_p} $ \cite[Theorem V.4.1 (c)]{Si}. If $ \lambda \in \overline{\F_p} $ is a root, then $ E^{\Leg}_{\lambda} $ is supersingular, and we again have $ \lambda \in \F_{p^2} $. Thus, there are exactly $ \frac{1}{2}(p-1) $ values of $ \lambda \in \F_{p^r} $ such that $ E^{\Leg}_{\lambda} $ is supersingular.
\end{proof}

\subsection{Universal family with a point of order \texorpdfstring{$4$}{4}}
\label{sub:Uni}

Consider the universal family of elliptic curves with a point of order $ 4 $, which will be called the \emph{$\Gamma_1(4)$-family}. It consists of elliptic curves $ E^{\Gamma_1(4)}_\lambda $ for $ \lambda \in \F_{p^r}\setminus\{0,4^{-1}\} $, which are given by the equation
\[
	E^{\Gamma_1(4)}_\lambda: y^2 = (x+\lambda)(x^2+x+\lambda).
\]
The point $ P_4 := (0, \lambda) $ has order $ 4 $ and $ 2P_4 $ is the point $ P_2 := (-\lambda, 0) $.

Katz \cite[p. 4]{Ka2}, in an observation attributed to Deligne, pointed out that the $\Gamma_1(4)$-family is isogenous to the Legendre family $ \Leg $. The 2-isogeny
\[
	\phi_{\lambda}: E^{\Gamma_1(4)}_{\lambda} \to E^{\Leg}_{1-4\lambda}
\]
has kernel equal to the group of order $ 2 $ generated by $ P_2 $. Following \cite{Ve}, it is possible to construct it explicitly.

Similarly as for the Legendre family, we let
\[
	N^{\Gamma_1(4)}(p^r;t) := \left|\left\{\lambda \in \F_{p^r}, \lambda \neq 0, 4^{-1}:\; \tr\left(E^{\Gamma_1(4)}_\lambda\right) = t\right\}\right|.
\]
Isogenous elliptic curves over $ \F_{p^r} $ have the same trace of Frobenius, hence $ N^{\Gamma_1(4)}(p^r;t) = N^{\Leg}(p^r;t) $ for every $ t \in \Z $. Building on work of Katz \cite{Ka2}, we will determine $ N^{\Gamma_1(4)}(p^r;t) $ directly by counting the elliptic curves with a point of order $ 4 $.

We begin by reviewing basic facts about quadratic orders and conductors. Let $ K $ be a quadratic field with the ring of integers $ \calO_K $. For a quadratic order $ \calO \subset \calO_K $, we let $ d(\calO) $ be the discriminant of $ \calO $. The \emph{conductor} $ f $ of $ \calO $ is defined as the index $ f := [\calO_K:\calO] $. The index is finite and $ \calO = \Z+f\calO_K $ \cite[Lemma 7.2]{Cox}. We will use the notation $ \calO(f) $ for a quadratic order of conductor $ f $. The discriminant and conductor of $ \calO = \calO(f) $ are related by the formula $ d(\calO) = f^2d(\calO_K) $.

If $ \calO_1 $, respectively $ \calO_2 $ are two orders in $ \calO_K $ with conductors $ f_1 $, respectively $ f_2 $, then
\[
	\calO_1 \subset \calO_2\qquad\Longleftrightarrow\qquad f_2 \mid f_1,
\]
and
\[
	d(\calO_1) = \frac{f_1^2}{f_2^2}d(\calO_2).
\]

If $ t \in \Z $, $ |t| \leq 2p^{\frac{r}{2}} $, $ p \nmid t $, then we let $ K $ be the quadratic field $ \Q(x)/(x^2-tx+p^r) $ and $ \calO_F $ the quadratic order $ \Z[x]/(x^2-tx+p^r) $ of discriminant $ D := t^2-4p^r $. Since $ d(\calO_K) $ is a fundamental discriminant, it is clear that the conductor $ f $ of $ \calO_F $ is the largest integer $ M $ such that $ M^2 \mid D $ and $ \frac{D}{M^2} \equiv 0, 1 \pmod{4} $ \cite[Lemma 3.4]{Ka2}.

Now let $ E/\F_{p^r} $ be an elliptic curve with $ \tr(E) = t $. The endomorphism ring $ \calO := \End_{\F_{p^r}}(E) $ can be viewed as a quadratic order $ \calO_F \subset \calO \subset \calO_K $. Let $ \left|\calM(p^r;t,\calO)\right| $ be the number of $ \F_{p^r} $-points of order $ 4 $ on $ E $. It is a crucial fact that $ \calM(p^r;t,\calO) $ depends only on $ t $ and $ \calO $, not $ E $. We have the following formula \cite[p. 9]{Ka2}:
\[
	2N^{\Gamma_1(4)}(p^r;t) = \sum_{\calO_F \subset \calO \subset \calO_K}\left|\calM(p^r;t,\calO)\right|\cdot h^*(d(\calO)),
\]
where the sum is taken over all quadratic orders $ \calO $ between $ \calO_F $ and $ \calO_K $. Note that $ N^{\Gamma_1(4)}(p^r;t) $ on the left-hand side is multiplied by $ 2 $ because the normalized class numbers $ h^*(d(\calO)) $ are weighted by half of the order of the automorphism group.

\begin{lemma}[{\cite[Lemma 3.6]{Ka2}}]
\label{lem:K1}
Let $ f $ be the conductor of $ \calO_F $ and write $ f = 2^af_0 $ for $ a \in \Z_{\geq 0} $ and $ f_0 \in \Z_{\geq 1} $ odd. Let $ b \in \Z $ such that $ 0 \leq b \leq a $ and $ f_1 \in \Z_{\geq 1} $ such that $ f_1 \mid f_0 $. Let $ t \in \Z $ such that $ |t| \leq 2p^{\frac{r}{2}} $ and $ p \nmid t $. If $ t \equiv p^r+5 \pmod{8} $, then
\[
	\left|\calM(p^r;t,\calO(2^bf_1))\right| = \begin{cases}
		2&\text{if $ b = a $,}\\
		0&\text{if $ 0 \leq b \leq a-1 $.}
	\end{cases}
\]
\end{lemma}
\begin{proof}
In \cite[Lemma 3.6]{Ka2}, a formula is given for $ \left|\calM(p^r;-A,\calO(2^bf_1))\right| $ under the assumption that $ p^r \equiv 1 \pmod{4} $ and $ p^r+1-A \equiv 0 \pmod{8} $ (note that our notation is slightly different). However, what is needed in the proof is only $ p^r+1+A \equiv 4 \pmod{8} $. We assume $ p^r+1-t \equiv 4 \pmod{8} $ and set $ A := -t $.
\end{proof}

\begin{lemma}
\label{lem:K2}
Under the assumptions of \Cref{lem:K1}, if $ t \equiv p^r+1 \pmod{8} $, then
\[
	\left|\calM(p^r;t,\calO(2^bf_1))\right| = \begin{cases}
		2&\text{if $ b = a $,}\\
		4&\text{if $ b = a-1 $,}\\
		12&\text{if $ 0 \leq b \leq a-2 $,}
	\end{cases}.
\]
\end{lemma}
\begin{proof}
The formula follows from \cite[Lemma 3.1 and Lemma 3.3]{Ka2} as in the proof of \cite[Theorem 2.8]{Ka2}.
\end{proof}

\begin{proposition}
\label{prop:G}
Let $ p \geq 5 $ be a prime, $ r \in \Z_{\geq 1} $, $ t \in \Z $ such that $ |t|\leq 2p^{\frac{r}{2}} $, and $ D := t^2-4p^r $. If $ p \nmid t $, then
\[
	N^{\Gamma_1(4)}(p^r;t) = \begin{cases}
		H\left(|D|\right)+H\left(\frac{|D|}{4}\right)+4H\left(\frac{|D|}{16}\right)&\text{if $ t \equiv p^r+1 \pmod{8} $,}\\
		H\left(|D|\right)-H\left(\frac{|D|}{4}\right)&\text{if $ t \equiv p^r+5 \pmod{8} $.}
	\end{cases}
\]
\end{proposition}
\begin{proof}
Let $ f $ be the conductor of $ \calO_F $. We again write $ f = 2^af_0 $ for $ a \in \Z_{\geq 0} $ and $ f_0 \in \Z_{\geq 1} $ odd.

The orders $ \calO_F \subset \calO \subset \calO_K $ correspond to conductors $ g \mid f $, and in this situation, the discriminant of $ \calO $ equals $ d(\calO) = \frac{g^2}{f^2}D $, where $ D $ is the discriminant of $ \calO_F $. Thus,
\begin{align*}
	2N^{\Gamma_1(4)}(p^r;t)& = \sum_{g \mid f}\left|\calM(p^r;t,\calO(g))\right|\cdot h^*\left(\frac{g^2D}{f^2}\right).
\end{align*}

Consider the case $ t \equiv p^r+5 \pmod{8} $. If $ g \mid f $, then we let $ g = 2^bf_1 $, where $ b \leq a $ and $ f_1 \mid f_0 $. By \Cref{lem:K1}, $ \left|\calM(p^r;t,\calO(2^af_1))\right| = 2 $ and $ \left|\calM(p^r;t,\calO(2^bf_1))\right| = 0 $ if $ b \leq a-1 $, hence
\[
	2N^{\Gamma_1(4)}(p^r;t) = \sum_{f_1 \mid f_0}2h^*\left(\frac{f_1^2D}{f_0^2}\right) = \sum_{g \mid f}2h^*\left(\frac{g^2D}{f^2}\right)-\sum_{2g \mid f}2h^*\left(\frac{g^2D}{f^2}\right).
\]
We recall that $ f $ is the largest integer $ M $ such that $ M^2 \mid D $ and $ \frac{D}{M^2} \equiv 0, 1 \pmod{4} $. It follows that
\begin{align*}
	N^{\Gamma_1(4)}(p^r;t)& = \sum_{\substack{M^2\mid D\\\frac{D}{M^2} \equiv 0,1 \pmod{4}}}h^*\left(\frac{D}{M^2}\right)-\sum_{\substack{M^2\mid D\\\frac{D}{M^2} \equiv 0,1 \pmod{4}\\2\mid M}}h^*\left(\frac{D}{M^2}\right)\\
	& = \sum_{\substack{M^2\mid D\\\frac{D}{M^2} \equiv 0,1 \pmod{4}}}h^*\left(\frac{D}{M^2}\right)-\sum_{\substack{M_1^2\mid \frac{D}{4}\\\frac{D}{4M_1^2} \equiv 0,1 \pmod{4}}}h^*\left(\frac{D}{4M_1^2}\right)
\end{align*}
Using the expression for the Hurwitz class number from \Cref{subsec:h}, we get
\[
	N^{\Gamma_1(4)}(p^r;t) = H(|D|)-H\left(\frac{|D|}{4}\right).
\]
The case $ t \equiv p^r+1 \pmod{8} $ is proved similarly using \Cref{lem:K2}.
\end{proof}

\begin{remark}
\label{rem:G}
We note that if $ p^r \equiv 3 \pmod{4} $ and $ t \equiv p^r+1 \pmod{4} $, then a short calculation shows that $ D $ is not divisible by $ 16 $, hence $ H\left(\frac{|D|}{16}\right) = 0 $.
\end{remark}

\subsection{An equivalent formula for the number of curves}

We eliminate $ H(|D|) $ from the formula using the relations between class numbers in \Cref{lem:delta2}.

\begin{proposition}
\label{prop:N}
Let $ p \geq 5 $ be a prime, $ r \in \Z_{\geq 1} $, $ t \in \Z $ such that $ |t| \leq 2p^{\frac{r}{2}} $, and $ D := t^2-4p^r $. If $ p \nmid t $, then
\[
	N^{\Leg}(p^r;t) = \begin{cases}
		3H\left(\frac{|D|}{4}\right)&\text{if $ t \equiv p^r+1 \pmod{4} $, $ p^r \equiv 3 \pmod{4} $,}\\
		4H\left(\frac{|D|}{4}\right)+2H\left(\frac{|D|}{16}\right)&\text{if $ t \equiv p^r+1 \pmod{8} $, $ p^r \equiv 1 \pmod{4} $,}\\
		2H\left(\frac{|D|}{4}\right)-2H\left(\frac{|D|}{16}\right)&\text{if $ t \equiv p^r+5 \pmod{8} $, $ p^r \equiv 1 \pmod{4} $.}
	\end{cases}
\]
\end{proposition}
\begin{proof}
First, let $ p^r \equiv 3 \pmod{4} $. As we observed in \Cref{rem:G}, if $ t \equiv p^r+1 \pmod{4} $, then $ H\left(\frac{|D|}{16}\right) = 0 $. We claim that
\[
	H(4p^r-t^2) = \begin{cases}
		2H\left(\frac{4p^r-t^2}{4}\right)&\text{if $ t \equiv p^r+1 \pmod{8} $,}\\
		4H\left(\frac{4p^r-t^2}{4}\right)&\text{if $ t \equiv p^r+5 \pmod{8} $.}
	\end{cases}
\]
It is straightforward to prove this using \Cref{lem:delta2}.
%%%%%%%%%%%%%%%%%%%%%%%%%%%%%%%%
%To prove the claim, let $ D_1 := \frac{t^2-4p^r}{4} $ and $ a \in \{0,4\} $ be such that $ t \equiv p^r+1+a \pmod{8} $. A direct computation shows $ D_1 \equiv 1+\frac{a^2}{4} \pmod{8} $.
%In fact, if we write $ t = 8k+p^r+1+a $ for some $ k \in \Z $, then
%\[
%	t^2-4p^r = 64k^2+16k(p^r+1+a)+(p^r+1+a)^2-4p^r \equiv (p^r-1)^2+2(p^r+1)a+a^2 \pmod{32}.
%\]
%Since $ p^r = 4\ell+3 $ for some $ \ell \in \Z_{\geq 0} $, we have
%\[
%	D_1 = \frac{t^2-4p^r}{4} \equiv 4\ell^2+4\ell+1+2(\ell+1)a+\frac{a^2}{4} \equiv 1+\frac{a^2}{4} \pmod{8}.
%\]
%In particular, $ D_1 $ is odd for both $ a = 0 $ and $ a = 4 $. Let $ D_1 = \Delta f^2 $, where $ \Delta < 0 $ is a fundamental discriminant and $ f \in \Z_{\geq 1} $. If we write $ f = 2^{\alpha}g $ with $ \alpha \in \Z_{\geq 0} $ and $ g \in \Z_{\geq 1} $ odd, then $ \alpha = 0 $. Thus, $ \Delta g^2 = D_1 $ and by \Cref{lem:delta2},
%\[
%	H(4|D_1|)=2H(|D_1|)+\left(1-\leg{\Delta}{2}\right)H(|D_1|).
%\]
%If $ a = 0 $, then $ \Delta \equiv D_1 \equiv 1 \pmod{8} $, hence $ \leg{\Delta}{2} = 1 $ and $ H(4|D_1|) = 2H(|D_1|) $. If $ a = 4 $, then $ \Delta \equiv D_1 \equiv 5 \pmod{8} $, hence $ \leg{\Delta}{2} = -1 $ and $ H(4|D_1|) = 4H(|D_1|) $. This proves the claim.
%%%%%%%%%%%%%%%%%%%%%%%%%%%%%%%%

Secondly, let $ p^r \equiv 1 \pmod{4} $. We claim that if $ t \equiv p^r+1 \pmod{4} $, then
\[
	H(4p^r-t^2) = 3H\left(\frac{4p^r-t^2}{4}\right)-2H\left(\frac{4p^r-t^2}{16}\right).
\]
This is again not difficult to prove using \Cref{lem:delta2}.
%%%%%%%%%%%%%%%%%%%%%%%%%%%%%%%%%
%We note that $ \frac{4p^r-t^2}{16} $ is an integer. If we write $ t = 4k+p^r+1 $ for some $ k \in \Z $, then
%\[
%	t^2-4p^r = 16k^2+8k(p^r+1)+(p^r-1)^2,
%\]
%and since $ p^r \equiv 1 \pmod{4} $, this is divisible by $ 16 $.
%
%Now we distinguish two cases: either $ \frac{t^2-4p^r}{16} $ is a discriminant or not. If it is a discriminant, then set $ D_1 := \frac{t^2-4p^r}{16} $ and let $ D_1 = \Delta f_1^2 $ as before. If we write $ f_1 = 2^{\alpha}g $ for $ \alpha \in \Z_{\geq 0} $ and $ g \in \Z_{\geq 1} $ odd, then $ 4D_1 = \Delta f_2^2 $, where $ f_2 = 2^{\alpha+1}g $. By \Cref{lem:delta2},
%\begin{align*}
%	H(4|D_1|)& = 2H(|D_1|)+\left(1+\leg{\Delta}{2}\right)H(|\Delta|g^2),\\
%	H(16|D_1|)& = 2H(4|D_1|)+\left(1+\leg{\Delta}{2}\right)H(|\Delta|g^2).
%\end{align*}
%From these two equations, it follows that $ H(16|D_1|) = 3H(4|D_1|)-2H(|D_1|) $.
%
%If $ \frac{t^2-4p^r}{16} $ is not a discriminant, then set $ D_1 := \frac{t^2-4p^r}{4} $ and again let $ D_1 = \Delta f^2 $ and $ f = 2^\alpha g $. Since $ D_1 $ is divisible by $ 4 $ but $ \frac{D_1}{4} $ is not a discriminant, $ f $ is odd and $ 4 \mid \Delta $. Hence, $ \alpha = 0 $, $ g = f $, and by \Cref{lem:delta2},
%\[
%	H(4|D_1|) = 2H(|D_1|)+\left(1+\leg{\Delta}{2}\right)H(|D_1|) = 3H(|D_1|) = 3H(|D_1|)-2H\left(\frac{|D_1|}{4}\right).
%\]
%This proves the claim.
%%%%%%%%%%%%%%%%%%%%%%%%%%%%%%%%%%%
The conclusion follows from \Cref{prop:G}.
\end{proof}

In the proposition below and elsewhere, $ \delta_{\calS} := 1 $ if the statement $ \calS $ is true and $ \delta_{\calS} := 0 $ otherwise.

\begin{proposition}[{\cite[Proposition 2.11]{OSS}}]
\label{prop:OSS}
Let $ p \geq 5 $ be a prime, $ r \in \Z_{\geq 1} $, and $ k \in \Z_{\geq 1} $.
\begin{enumerate}
\item If $ k $ is even, then
\[
	S_k^{\Leg}(p^r) = 2+\delta_{2\mid r}2^{k-1}p^{\frac{rk}{2}}(p-1)+3\sum_{\substack{p \nmid t\\t \equiv p^r+1 \pmod{4}}}H\left(\frac{4p^r-t^2}{4}\right)t^k.
\]
\item If $ k $ is odd and $ p^r \equiv 3 \pmod{4} $, then $ S_k^{\Leg}(p^r) = 0 $.
\item If $ k $ is odd and $ p^r \equiv 1 \pmod{4} $, then
\begin{align*}
	S_k^{\Leg}(p^r) =& 2+\delta_{2\mid r}\varepsilon(p^r)2^{k-1}p^{\frac{rk}{2}}(p-1)+2\sum_{\substack{p \nmid t\\t \equiv p^r+1 \pmod{8}}}H\left(\frac{4p^r-t^2}{4}\right)t^k\\
	&+4\sum_{\substack{p \nmid t\\t \equiv p^r+1 \pmod{8}}}H\left(\frac{4p^r-t^2}{16}\right)t^k.
\end{align*}
\end{enumerate}
\end{proposition}

\begin{remark}
\label{rem:OSS}
If $ p^r \equiv 1 \pmod{4} $ and $ t \equiv p^r+9 \pmod{16} $, then $ 16 \mid (4p^r-t^2) $ and a short calculation shows that
\[
	\frac{t^2-4p^r}{16} \equiv 2, 3 \pmod{4},
\]
hence $ \frac{t^2-4p^r}{16} $ is not a discriminant and $ H\left(\frac{4p^r-t^2}{16}\right) = 0 $. Thus, in \Cref{prop:OSS} (3), the condition $ t \equiv p^r+1 \pmod{8} $ in the second sum can be replaced by $ t \equiv p^r+1 \pmod{16} $ as in \cite[Proposition 2.11]{OSS}.
\end{remark}

\begin{proof}[Proof of \Cref{prop:OSS}]
It is straightforward to show by elementary counting that we have $ \tr\left(E^{\Leg}_0\right) = 1 $ and $ \tr\left(E^{\Leg}_1\right) = \leg{-1}{p^r} $ for the singular curves.
%%%%%%%%%%%%%%%%%%%%%%%%%%%%%%%%%%%%%%
%First, we count the number of solutions to $ y^2 = x^2(x-1) $. One solution is $ (0,0) $. If $ x \in \F_{p^r}\setminus\{0\} $, then the equation is equivalent to
%\[
%	\left(\frac{y}{x}\right)^2 = x-1
%\]
%and the number of solutions is $ 1+\phi_{p^r}(x-1) $, where $ \phi_{p^r} $ is the quadratic character of $ \F_{p^r} $. Thus,
%\[
%	\left|\left\{(x,y)\in \F_{p^r}^2:\; y^2 = x^2(x-1)\right\}\right| = 1+\sum_{x \in \F_{p^r}\setminus\{0\}}\left(1+\phi_{p^r}(x-1))\right) = p^r-\phi_{p^r}(-1).
%\]
%Counting the point at infinity, we have $ \left|E^{\Leg}_0(\F_{p^r})\right| = p^r+1-\phi_{p^r}(-1) $, hence
%\[
%	\tr\left(E^{\Leg}_0\right) = \phi_{p^r}(-1) = \leg{-1}{p^r}.
%\]
%
%Secondly, we count the number of solutions to $ y^2 = x(x-1)^2 $. One solution is $ (1,0) $. If $ x \in \F_{p^r}\setminus\{1\} $, then the equation is equivalent to
%\[
%	\left(\frac{y}{x-1}\right)^2 = x
%\]
%and the number of solution is $ 1+\phi_{p^r}(x) $, hence
%\[
%	\left|\left\{(x,y)\in \F_{p^r}^2:\; y^2 = x(x-1)^2\right\}\right| = 1+\sum_{x \in \F_{p^r}\setminus\{1\}}(1+\phi_{p^r}(x)) = p^r-1.
%\]
%Counting the point at infinity, we have $ \left|E^{\Leg}_1(\F_{p^r})\right| = p^r $, hence
%\[
%	\tr\left(E^{\Leg}_1\right) = 1.
%\]
%%%%%%%%%%%%%%%%%%%%%%%%%%%%%%%%%%%%%
Thus,
\[
	S_k^{\Leg}(p^r) = \sum_{\lambda \in \F_{p^r}}\tr\left(E^{\Leg}_{\lambda}\right)^k = 1+\leg{-1}{p^r}^k+\sum_{\lambda \in \F_{p^r}\setminus\{0,1\}}\tr\left(E^{\Leg}_{\lambda}\right)^k.
\]
The number of supersingular curves was evaluated in \Cref{prop:SS}, which yields
\[
	S_k^{\Leg}(p^r) = 1+\leg{-1}{p^r}^k+\delta_{2\mid r}\frac{1}{2}(p-1)\left(\varepsilon(p^r)2p^{\frac{r}{2}}\right)^k+\sum_{p \nmid t}N^{\Leg}(p^r;t)t^k.
\]
The rest now follows from \Cref{prop:N}.
%%%%%%%%%%%%%%%%%%%%%%%%%%%%%%%%%%%%%
%Let us prove for example the case of $ k $ even. If $ p^r \equiv 3 \pmod{4} $, then
%\[
%	\sum_{p \nmid t}N^\Leg(p^r;t)t^k = \sum_{\substack{p \nmid t\\t \equiv p^r+1 \pmod{4}}}3H\left(\frac{4p^r-t^2}{4}\right)t^k.
%\]
%If $ p^r \equiv 1 \pmod{4} $, then $ t \equiv p^r+1 \pmod{8} $ if and only if $ -t \equiv p^r+5 \pmod{8} $, hence
%\begin{align*}
%	& \sum_{p \nmid t}N^\Leg(p^r;t)t^k = \sum_{\substack{p \nmid t\\t \equiv p^r+1 \pmod{8}}}\left(4H\left(\frac{4p^r-t^2}{4}\right)+2H\left(\frac{4p^r-t^2}{16}\right)\right)t^k\\
%	& +\sum_{\substack{p \nmid t\\t \equiv p^r+5 \pmod{8}}}\left(2H\left(\frac{4p^r-t^2}{4}\right)-2H\left(\frac{4p^r-t^2}{16}\right)\right)t^k = \sum_{\substack{p \nmid t\\t \equiv p^r+1 \pmod{8}}}6H\left(\frac{4p^r-t^2}{4}\right)t^k\\
%	& = \sum_{\substack{p \nmid t\\t \equiv p^r+1 \pmod{4}}}3H\left(\frac{4p^r-t^2}{4}\right)t^k.
%\end{align*}
%The case of $ k $ odd is similar.
%%%%%%%%%%%%%%%%%%%%%%%%%%%%%%%%%%%%
\end{proof}

\section{Holomorphic projection}
\label{sec:hol}

\subsection{Definition and properties}

Gross and Zagier \cite[p. 288, Proposition 5.1]{GZ}) defined holomorphic projection. The definition given here is stated under the same assumptions as in \cite[Section 2.3]{BKP}.

Let $ F(\tau) = \sum_{n \in \Z}c_{F,v}(n)q^n $ be a (not necessarily holomorphic) function on $ \H $ such that the series converges absolutely for every $ \tau \in \H $. Assume that
\begin{enumerate}
\item $ F $ is translation-invariant, i.e., $ F(\tau+1) = F(\tau) $,
\item $ c_{F,v}(n) = O_{F,n}\left(v^{2-\kappa}\right) $ as $ v \to 0^+ $ for $ n \in \Z_{\geq 1} $,
\item there exists $ c_F(0) \in \C $ such that
\[
	F(\tau) = c_F(0)+O\left(v^{-\delta}\right),\qquad \tau \to i\infty
\]
for some $ \delta>0 $, and a similar growth condition holds as $ \tau \to \rho $ for every cusp $ \rho \in \Q $.
\end{enumerate}
Then the \emph{holomorphic projection} of $ F $ is defined by
\[
	(\pihol(F))(\tau) := c_F(0)+\sum_{n=1}^\infty c(n)q^n,
\]
where
\[
	c(n) := \frac{(4\pi n)^{\kappa-1}}{\Gamma(\kappa-1)}\int_0^{\infty} c_{F,v}(n)e^{-4\pi n v}v^{\kappa-2}\dif v.
\]

If $ F $ satisfies weight $ \kappa $ modularity on a congruence group $ \Gamma $ and $ \langle F,g \rangle $ is defined for every cusp form $ g $ of weight $ \kappa $ on $ \Gamma $, where $ \langle \cdot, \cdot \rangle $ denotes the \emph{Petersson inner product}, then there exists a unique cusp form $ f $ such that $ \langle F, g \rangle = \langle f, g \rangle $ for every cusp form $ g $. This observation is due to Sturm \cite{Stu}, who originally defined $ \pihol(F) = f $.

\begin{definition}
\label{def:RC}
For two functions $ F_1, F_2 $ and $ \kappa_1, \kappa_2 \in \R $, define for $ k \in \Z_{\geq 0} $ the $ k $-th \emph{Rankin-Cohen bracket} by
\[
	[F_1, F_2]_k := \frac{1}{(2\pi i)^k}\sum_{j=0}^k (-1)^j\binom{\kappa_1+k-1}{k-j}\binom{\kappa_2+k-1}{j}\frac{\partial^j F_1(\tau)}{\partial \tau^j}\frac{\partial^{k-j} F_2(\tau)}{\partial \tau^{k-j}}
\]
with $ \binom{\alpha}{j} := \frac{\Gamma(\alpha+1)}{j!\Gamma(\alpha-j+1)} $, provided that the right-hand side is well defined. If $ F_1, F_2 $ transform like modular forms of weight $ \kappa_1 $ and $ \kappa_2 $, respectively, then $ [F_1, F_2]_k $ transforms like a modular form of weight $ \kappa_1+\kappa_2+2k $.
\end{definition}

A harmonic Maass form of weight $ \kappa \neq 1 $ has a decomposition \cite[Lemma 7.2]{On2}
\begin{equation}
\label{eq:Fdec}
	F(\tau) = F^{+}(\tau)+F^{-}(\tau),
\end{equation}
where
\begin{align*}
	F^+(\tau)& = \sum_{n=n_1}^\infty c_F^{+}(n)q^n,\\
	F^-(\tau)& = c_F^-(0)v^{1-\kappa}+\sum_{\substack{n=-\infty\\n \neq 0}}^{n_2} c_F^{-}(n)\Gamma(1-\kappa,-4\pi n v)q^n
\end{align*}
for some $ n_1, n_2 \in \Z $. Here
\[
	\Gamma(\alpha, x) := \int_x^\infty t^{\alpha-1}e^{-t}\dif t
\]
is the \emph{incomplete gamma function}. We also let
\[
	F_0^-(\tau) := F^-(\tau)-c_F^-(0)v^{1-\kappa}.
\]

Properties of the holomorphic projection are summarized in the next lemma.

\begin{lemma}[{\cite[Proposition 4.2]{Me2}, \cite[Lemma 2.6]{KP1}}]
\label{lem:holProp}
Suppose that $ F $ satisfies the conditions from the definition of holomorphic projection and weight $ \kappa \geq 2 $ modularity on $ \Gamma_1(N) $.
\begin{enumerate}
\item If $ F $ is holomorphic, then $ \pihol(F) = F $.
\item If $ \kappa > 2 $, then $ \pihol(F) \in M_k(\Gamma) $. If $ \kappa = 2 $, then $ \pihol(F) $ is a quasimodular form of weight $ 2 $.
\item If $ F_1 $ is a weight $ \kappa_1 \in \frac{1}{2}\Z_{\geq 1} $ harmonic Maass form and $ f_2 $ is a weight $ \kappa_2 \in \frac{1}{2}\Z_{\geq 1} $ holomorphic modular form, and $ k \in \Z_{\geq 0} $ with $ \kappa := \kappa_1+\kappa_2+2k \geq 2 $, then
\[
	\pihol([F_1, f_2]_k) = [F_1^{+}, f_2]_k+\pihol([F_1^{-}, f_2]_k).
\]
\end{enumerate}
\end{lemma}

\subsection{The non-holomorphic part}

For $ a \in \Z_{\geq 1} $ and $ b \in \R $, let
\[
	P_{a,b}(X,Y) := \sum_{j=0}^{a-2}\binom{j+b-2}{j}X^j(X+Y)^{a-j-2}.
\]
Moreover, for $ \kappa_1, \kappa_2 \in \R\setminus \Z $ and $ k \in \Z_{\geq 0} $ with $ 2k-2 \geq \kappa_1+\kappa_2 \in \Z $, let
\begin{multline*}
	\alpha_{\kappa_1,\kappa_2,k} := \frac{1}{(\kappa_1+\kappa_2+2k-2)!(\kappa_1-1)}\\\cdot\sum_{\mu=0}^k\frac{\Gamma(2-\kappa_1)\Gamma(\kappa_2+2k-\mu)}{\Gamma(2-\kappa_1-\mu)}\binom{\kappa_1+k-1}{k-\mu}\binom{\kappa_2+k-1}{\mu}.
\end{multline*}

In \cite[Lemma 4.4, Theorem 4.6]{Me2}, Mertens explicitly worked out $ \pihol([F_1^{-}, f_2]_k) $. His results were summarized in \cite[Lemma 2.1]{BKP}, which we restate here.

\begin{lemma}[{\cite[Lemma 2.1]{BKP}}]
\label{lem:holGen}
Let $ \kappa_1, \kappa_2 \in \Z+\frac{1}{2} $, $ \kappa_1, \kappa_2 \geq 0 $, and $ k \in \Z_{\geq 0} $ with $ 2k-2 \geq \kappa_1+\kappa_2 $. Suppose that $ F $ satisfies modularity of weight $ \kappa_1 $, has Fourier expansion of the type \eqref{eq:Fdec}, and grows at most polynomially towards the cusps and that $ g $ is a holomorphic modular form of weight $ \kappa_2 $.
\begin{enumerate}
\item We have
\[
	\frac{(4\pi)^{1-\kappa_1}}{\kappa_1-1}\pihol\left([v^{1-\kappa_1},g]_k\right) = \alpha_{\kappa_1,\kappa_2,k}\sum_{n=0}^{\infty}n^{\kappa_1+k-1}c_g(n)q^n.
\]
\item We have
\[
	\pihol\left([F_0^{-},g]_k\right)(\tau) = \sum_{n=1}^{\infty}b(n)q^n,
\]
where
\begin{multline*}
	b(n) = -\Gamma(1-\kappa_1)\sum_{\substack{j \geq 1, \ell \geq 1\\j-\ell=n}}\sum_{\mu=0}^k\binom{\kappa_1+k-1}{k-\mu}\binom{\kappa_2+k-1}{\mu}j^{k-\mu}c_g(j)\frac{c_F^{-}(-\ell)}{\ell^{\kappa_1-1}}\\
	\cdot \left(j^{\mu-2k-\kappa_2+1}P_{\kappa_1+\kappa_2+2k,2-\kappa_1-\mu}(n,\ell)-\ell^{\kappa_1+\mu-1}\right).
\end{multline*}
\end{enumerate}
\end{lemma}

Next, we apply \Cref{lem:holGen} in two situations: first, when $ F $ is of weight $ \frac{3}{2} $ and $ g $ is of weight $ \frac{1}{2} $, and secondly, when $ F $ and $ g $ are both of weight $ \frac{3}{2} $. For $ \ell \in \Z_{\geq 0} $, $ a, b \in \Z $, and Dirichlet characters $ \chi $ and $ \psi $, let
\[
	\Lambda_{\ell, a, b}^{\chi,\psi}(\tau) := 2\sum_{n=1}^\infty \lambda_{\ell,a,b}^{\chi,\psi}(n)q^n,\qquad\text{where}\qquad \lambda_{\ell,a,b}^{\chi,\psi}(n) := \sideset{}{^*}\sum_{\substack{t \geq 1, s \geq 0\\at^2-bs^2=n}}\chi(t)\overline{\psi(s)}\left(t\sqrt{a}-s\sqrt{b}\right)^\ell
\]
and the $ ^* $ next to the sum means that the terms with $ s = 0 $ are weighted by $ \frac{1}{2} $.

\begin{proposition}
\label{prop:hol1}
Suppose that $ \chi $ and $ \psi $ are even characters of conductors $ N_{\chi} $ and $ N_{\psi} $, respectively and $ N, a, b \in \Z_{\geq 1} $ with $ bN_{\psi}^2 \mid N $. If $ F $ is a harmonic Maass form of weight $ \frac{3}{2} $ on $ \Gamma_1(4N) $ that grows at most polynomially towards all cusps and satisfies $ \xi_{\frac{3}{2}}(F) = \theta_{\psi}\vert V_b $, then
\[
	\pihol\left([F^{-},\theta_{\chi}\vert V_a]_k\right) = -2^{3-2k}\pi\binom{2k}{k}\Lambda^{\chi,\psi}_{2k+1,a,b}.
\]
Moreover, $ \pihol\left([F,\theta_{\chi}\vert V_a]_k\right) $ is a holomorphic cusp form of weight $ 2k+2 $ on $ \Gamma_1(\lcm(4N, 4aN_{\chi}^2)) $ if $ k > 0 $ and a quasimodular form of weight $ 2 $ if $ k = 0 $.
\end{proposition}
\begin{proof}
This is essentially \cite[Lemma 2.7]{KP1}, which is however stated for $ [F,\theta_{\chi}\vert V_a]_k \vert U_4 $. The proof follows from \Cref{lem:holProp} and \Cref{lem:holGen} using the evaluation of $ \alpha_{\frac{3}{2},\frac{1}{2},k} $ from \cite[Lemma 5.2]{Me2} and \cite[Proposition 5.3]{Me2}.
\end{proof}

We will prove a similar result with $ \theta_{\chi} $ replaced by $ \theta_{1,\chi} $ in a series of lemmas. To evaluate the coefficient $ \alpha_{\frac{3}{2},\frac{3}{2},k} $, we follow the approach of Mertens \cite[Lemma 5.2]{Me2} (see also \cite[Lemma V.2.6]{Me1}), who proved
\[
	\alpha_{\frac{3}{2},\frac{1}{2},k} = 2^{1-2k}\sqrt{\pi}\binom{2k}{k}.
\]

\begin{lemma}
\label{lem:alpha}
If $ k \in \Z_{\geq 0} $, then
\[
	\alpha_{\frac{3}{2},\frac{3}{2},k} = 2^{1-2k}\sqrt{\pi}\frac{2k+1}{2k+2}\binom{2k}{k}.
\]
\end{lemma}
\begin{proof}
The key steps are expressing the coefficient in terms of a certain generalized hypergeometric series and using the Pfaff-Saalsch\"{u}tz identity \cite[Theorem 2.2.6]{AAR}. We omit the details.

\end{proof}

\begin{lemma}[{\cite[Proposition 3.3]{Me2}}]
\label{lem:xi}
Let $ F $ be a harmonic Maass form of weight $ \kappa $. We have
\[
	\left(\xi_{\kappa}F\right)(\tau) = (1-\kappa)\overline{c_F^{-}(0)}-\sum_{n=1}^{\infty}\frac{(4\pi)^{1-\kappa}}{n^{\kappa-1}}\overline{c_F^{-}(-n)}q^n.
\]
\end{lemma}
We remark that statement of the lemma slightly differs from \cite[Proposition 3.3]{Me2} because of differences in the definition of $ F^{-} $.

\begin{lemma}[{\cite[Proposition 4.2]{OSS}}]
\label{lem:P}
If $ k \in \Z_{\geq 0} $ and $ m, n \in \Z_{\geq 1} $ with $ m > n $, then
\begin{multline*}
	2^{-2k}\frac{2k+1}{2k+2}\binom{2k}{k}\left(\sqrt{m}-\sqrt{n}\right)^{2k+2}\\
	=\sum_{\mu=0}^k\binom{k+\frac{1}{2}}{k-\mu}\binom{k+\frac{1}{2}}{\mu}m^{k-\mu+\frac{1}{2}}\cdot\left(m^{\mu-2k-\frac{1}{2}}P_{3+2k,\frac{1}{2}-\mu}(m-n,n)-n^{\frac{1}{2}+\mu}\right).
\end{multline*}
\end{lemma}

\begin{lemma}
\label{lem:bn}
Let $ \chi $ be an odd character of conductor $ N_{\chi} $ and $ \psi $ an even character of conductor $ N_{\psi} $. Let $ N, a, b \in \Z_{\geq 1} $ with $ bN_{\psi}^2\mid N $. If $ F $ is a harmonic Maass form of weight $ \frac{3}{2} $ on $ \Gamma_1(4N) $ that grows at most polynomially towards all cusps and satisfies $ \xi_{\frac{3}{2}}(F) = \theta_{\psi}\vert V_b $, then
\[
	\pihol\left([F_0^{-},\theta_{1,\chi}\vert V_a]_k\right) = \sum_{n=1}^{\infty}b(n)q^n,
\]
where
\[
	b(n) = -2^{4-2k}\frac{\pi}{\sqrt{a}}\frac{2k+1}{2k+2}\binom{2k}{k}\sum_{\substack{t \geq 1, s \geq 1\\at^2-bs^2=n}}\chi(t)\overline{\psi(s)}\left(\sqrt{a}t-\sqrt{b}s\right)^{2k+2}.
\]
\end{lemma}
\begin{proof}
We apply \Cref{lem:holGen} with $ \kappa_1 = \kappa_2 = \frac{3}{2} $ and
\[
	g = \theta_{1,\chi}\vert V_a = \sum_{t=1}^\infty 2\chi(t)tq^{at^2}.
\]
Thus, we have
\[
	c_g(j) = \begin{cases}
		2\chi(t)t&\text{if $ j = at^2 $, $ t \in \Z_{\geq 1} $,}\\
		0&\text{otherwise.}
	\end{cases}
\]
Since we assume
\[
	\xi_{\frac{3}{2}}(F) = \theta_{\psi}\vert V_b = \sum_{s \in \Z}\psi(s)q^{bs^2} = \psi(0)+\sum_{s=1}^{\infty}2\psi(s)q^{bs^2},
\]
we get
\begin{equation}
\label{eq:cF}
	c_F^{-}(0) = -2\overline{\psi(0)},\qquad c_F^{-}(-\ell) = \begin{cases}
	-(4\pi)^{\frac{1}{2}}\sqrt{b}s2\overline{\psi(s)}&\text{if $ \ell = bs^2 $, $ s \in \Z_{\geq 1} $,}\\
	0&\text{otherwise}
	\end{cases}
\end{equation}
from \Cref{lem:xi}. By \Cref{lem:holGen} (2),
\begin{align*}
	b(n) = &-\Gamma\left(-\frac{1}{2}\right)\sum_{\substack{j \geq 1, \ell \geq 1\\j-\ell=n}}\sum_{\mu=0}^k\binom{k+\frac{1}{2}}{k-\mu}\binom{k+\frac{1}{2}}{\mu}j^{k-\mu}c_g(j)\frac{c_F^{-}(-\ell)}{\ell^{\frac{1}{2}}}\\
	&\cdot\left(j^{\mu-2k-\frac{1}{2}}P_{3+2k,\frac{1}{2}-\mu}(n,\ell)-\ell^{\frac{1}{2}+\mu}\right)\\
	= &-\Gamma\left(-\frac{1}{2}\right)\sum_{\substack{t \geq 1, s \geq 1\\at^2-bs^2=n}}\sum_{\mu=0}^k\binom{k+\frac{1}{2}}{k-\mu}\binom{k+\frac{1}{2}}{\mu}(at^2)^{k-\mu}2\chi(t)t\left(-(4\pi)^{\frac{1}{2}}\right)2\overline{\psi(s)}\\
	&\cdot\left((at^2)^{\mu-2k-\frac{1}{2}}P_{3+2k,\frac{1}{2}-\mu}(n,bs^2)-(bs^2)^{\frac{1}{2}+\mu}\right).
\end{align*}
Next, from $ \Gamma\left(-\frac{1}{2}\right) = -2\sqrt{\pi} $ we get
\begin{align*}
	b(n) = &-\frac{16\pi}{\sqrt{a}}\sum_{\substack{t \geq 1, s \geq 1\\at^2-bs^2=n}}\chi(t)\overline{\psi(s)}\sum_{\mu=0}^k\binom{k+\frac{1}{2}}{k-\mu}\binom{k+\frac{1}{2}}{\mu}(\sqrt{a}t)^{2k-2\mu+1}\\
	&\cdot\left((\sqrt{a}t)^{2\mu-4k-1}P_{3+2k,\frac{1}{2}-\mu}(n,(\sqrt{b}s)^2)-(\sqrt{b}s)^{1+2\mu}\right).
\end{align*}
By \Cref{lem:P}, the inner sum equals
\[
	2^{-2k}\frac{2k+1}{2k+2}\binom{2k}{k}\left(\sqrt{a}t-\sqrt{b}s\right)^{2k+2},
\]
and the formula for $ b(n) $ follows.

Finally, for the proof to be complete, it must be checked that the sum
\[
	\sum_{\substack{t\geq 1, s\geq 1\\at^2-bs^2=n}}\chi(t)\overline{\psi(s)}\left(\sqrt{a}t-\sqrt{b}s\right)^{2k+2}
\]
converges. This can be done in exactly the same way as in \cite[p. 374]{Me2}.
\end{proof}

\begin{proposition}
\label{prop:hol2}
Let $ \chi $ be an odd character of conductor $ N_{\chi} $ and $ \psi $ an even character of conductor $ N_{\psi} $. Let $ N, a, b \in \Z_{\geq 1} $ with $ bN_{\psi}^2\mid N $. If $ F $ is a harmonic Maass form of weight $ \frac{3}{2} $ on $ \Gamma_1(4N) $ that grows at most polynomially towards all cusps and satisfies $ \xi_{\frac{3}{2}}(F) = \theta_{\psi}\vert V_b $, then
\[
	\pihol\left(\left[F^{-},\theta_{1,\chi}\vert V_a\right]_k\right)=-2^{3-2k}\frac{\pi}{\sqrt{a}}\frac{2k+1}{2k+2}\binom{2k}{k}\Lambda^{\chi,\psi}_{2k+2,a,b}.
\]
Moreover, $ \pihol\left(\left[F,\theta_{1,\chi}\vert V_a\right]_k\right) $ is a holomorphic cusp form of weight $ 2k+3 $ on the group $ \Gamma_1(\lcm(4N,4aN_{\chi}^2)) $.
\end{proposition}
\begin{proof}
We again apply \Cref{lem:holGen} with $ \kappa_1 = \kappa_2 = \frac{3}{2} $ and $ g = \theta_{1,\chi}\vert V_a $. As in the proof of \Cref{lem:bn}, we have
\[
	c_g(j) = \begin{cases}
		2\chi(t)t&\text{if $ j = at^2, t \in \Z_{\geq 1} $,}\\
		0&\text{otherwise}.
	\end{cases}
\]
First, using the evaluation of $ \alpha_{\frac{3}{2},\frac{3}{2},k} $ from \Cref{lem:alpha} in \Cref{lem:holGen} (1), we get
\[
	\pihol\left([v^{-\frac{1}{2}},\theta_{1,\chi}\vert V_a]_k\right) = 2^{2-2k}\frac{\pi}{\sqrt{a}}\frac{2k+1}{2k+2}\binom{2k}{k}\sum_{t=1}^{\infty}\chi(t)(\sqrt{a}t)^{2k+2}q^{at^2}.
\]
%%%%%%%%%%%%%%%%%%%%%%%%%%%%%%%%%%%%%%%%%%%%%%%%%%%%%%%%%%%%%%%%%%%%%%%%%
%Indeed,
%\[
%	\pi^{-\frac{1}{2}}\pihol\left([v^{-\frac{1}{2}},\theta_{1,\chi}\vert V_a]_k\right) = \alpha_{\frac{3}{2},\frac{3}{2},k}\sum_{t=1}^{\infty}(at^2)^{k+\frac{1}{2}}2\chi(t)tq^{at^2} = \frac{2\alpha_{\frac{3}{2},\frac{3}{2},k}}{\sqrt{a}}\sum_{t=1}^{\infty}(\sqrt{a}t)^{2k+2}\chi(t)q^{at^2}.
%\]
%%%%%%%%%%%%%%%%%%%%%%%%%%%%%%%%%%%%%%%%%%%%%%%%%%%%%%%%%%%%%%%%%%%%%%%%
Secondly, by \Cref{lem:bn},
\[
	\pihol\left([F_0^{-},\theta_{1,\chi}\vert V_a]_k\right) = -2^{4-2k}\frac{\pi}{\sqrt{a}}\frac{2k+1}{2k+2}\binom{2k}{k}\sum_{n=1}^{\infty}\sum_{\substack{t \geq 1, s \geq 1\\at^2-bs^2=n}}\chi(t)\overline{\psi(s)}\left(\sqrt{a}t-\sqrt{b}s\right)^{2k+2}q^n.
\]

We have
\[
	\pihol\left([F^{-},\theta_{1,\chi}\vert V_a]_k\right) = \pihol\left([c_F^{-}(0)v^{-\frac{1}{2}},\theta_{1,\chi}\vert V_a]_k\right)+\pihol\left([F_0^{-},\theta_{1,\chi}\vert V_a]_k\right)
\]
with $ c_F^{-}(0) = -2\overline{\psi(0)} $ by \eqref{eq:cF}. The formula follows by plugging in the expressions for the individual terms.

To prove the second part, we note that the Rankin-Cohen bracket $ [F,\theta_{1,\chi}\vert V_a]_k $ satisfies modularity of weight $ 2k+3 $ on $ \Gamma := \Gamma_1(\lcm(4N,4aN_{\chi}^2)) $. The holomorphic projection is a holomorphic modular form of weight $ 2k+3 $ on $ \Gamma $ by \Cref{lem:holProp}. The fact that it is a cusp form can be showed in exactly the same way as in the proof of \cite[Lemma 3.1]{BKP}.
\end{proof}

\subsection{Application to Hurwitz class numbers}

We apply \Cref{prop:hol1,prop:hol2} to $ F := \widehat{\calH}\vert V_b $, where $ \widehat{\calH} $ was defined as the completion of the generating function for the Hurwitz class numbers. The $ \xi $-operator of weight $ \frac{3}{2} $ maps $ \widehat{H} $ to $ \xi_{\frac{3}{2}}\left(\widehat{\calH}\right) = -\frac{1}{16\pi}\theta_0 $, and hence
\[
	\xi_{\frac{3}{2}}\left(\widehat{\calH}|V_b\right) = 2iv^{\frac{3}{2}}\overline{\frac{\partial}{\partial\overline{\tau}}}\widehat{\calH}(b\tau) = 2ib^{-\frac{1}{2}}(bv)^{\frac{3}{2}}\left(\overline{\frac{\partial}{\partial\overline{\tau}}}\widehat{\calH}\right)(b\tau) = -\frac{b^{-\frac{1}{2}}}{16\pi}\theta_0\vert V_b.
\]

For $ \ell \in \Z_{\geq 0} $, $ m \in \Z $, $ M \in \Z_{\geq 1} $, and $ b \in \Z_{\geq 1} $, let
\[
	\lambda_{\ell,m,M,b}(n) := \sum_{\varepsilon\in\{\pm 1\}}\varepsilon^\ell\sideset{}{^*}\sum_{\substack{t \geq 1, s \geq 0\\t^2-bs^2 = n\\t \equiv \varepsilon m \pmod{M}}}\left(t-\sqrt{b}s\right)^{\ell+1}
\]
and consider the generating function
\[
	\Lambda_{\ell,m,M,b}(\tau) := \sum_{n=1}^{\infty}\lambda_{\ell,m,M,b}(n)q^n.
\]

The following is similar to \cite[Lemma 2.2]{KPY}, which follows as a special case by setting $ b = 1 $ and applying the operator $ U_4 $.

\begin{proposition}
\label{prop:holH}
If $ k \in \Z_{\geq 0} $, $ m \in \Z $, $ M \in \Z_{\geq 1} $, $ b \in \Z_{\geq 1} $, and $ j \in \{0,1\} $ is such that $ j \equiv k \pmod{2} $, then
\[
	\pihol\left(\left[\left(\widehat{\calH}\vert V_b\right)^{-},\theta_{j,m,M}\right]_{\lfloor\frac{k}{2}\rfloor}\right) = b^{-\frac{1}{2}}2^{-1-k}\binom{k}{\left\lfloor k/2\right\rfloor}\Lambda_{k,m,M,b}.
\]
Moreover, $ \pihol\left(\left[\widehat{\calH}\vert V_b,\theta_{j,m,M}\right]_{\lfloor\frac{k}{2}\rfloor}\right) $ is a holomorphic cusp form of weight $ k+2 $ on $ \Gamma_{4M^2b,M} = \Gamma_0(4M^2b)\cap\Gamma_1(M) $ (respectively $ \Gamma_0(4M^2b) $) if $ M \nmid m $ (respectively $ M \mid m $) if $ k \geq 1 $ and quasimodular on that group if $ k = 0 $.
\end{proposition}
\begin{proof}
If we let $ g = \gcd(m,M) $, $ m_1 = \frac{m}{g} $, and $ M_1 = \frac{M}{g} $, then
\begin{align*}
	\theta_{j,m,M}(\tau)& = \sum_{t \equiv m \pmod{M}}t^jq^{t^2} = \frac{1}{2}\sum_{\varepsilon \in \{\pm 1\}}\varepsilon^j\sum_{t \equiv \varepsilon m \pmod{M}}t^jq^{t^2}\\
	& = \frac{g^j}{2}\sum_{\varepsilon \in \{\pm 1\}}\varepsilon^j\sum_{t \equiv \varepsilon m_1 \pmod{M_1}}t^jq^{g^2t^2}.
\end{align*}
From orthogonality of characters, we get for every $ t \in \Z $
\begin{align*}
	\frac{1}{\varphi(M_1)}\sum_{\chi \pmod{M_1}}\overline{\chi(m_1)}\chi(t)& = \delta_{t \equiv m_1 \pmod{M_1}},\\
	\frac{1}{\varphi(M_1)}\sum_{\chi \pmod{M_1}}\overline{\chi(-m_1)}\chi(t)& = \delta_{t \equiv -m_1 \pmod{M_1}}.
\end{align*}
Either summing or subtracting these two equalities together yields
\begin{equation}
\label{eq:holH}
	\frac{2}{\varphi(M_1)}\sum_{\substack{\chi \pmod{M_1}\\\chi(-1) = (-1)^j}}\overline{\chi(m_1)}\chi(t) = \sum_{\varepsilon\in\{\pm 1\}}\varepsilon^j\delta_{t \equiv \varepsilon m_1 \pmod{M_1}}.
\end{equation}
Hence,
\[
	\theta_{j,m,M}(\tau) = \frac{g^j}{\varphi(M_1)}\sum_{\substack{\chi \pmod{M_1}\\\chi(-1)=(-1)^j}}\overline{\chi(m_1)}\sum_{t \in \Z}\chi(t)t^jq^{g^2t^2} = \frac{g^j}{\varphi(M_1)}\sum_{\substack{\chi\pmod{M_1}\\\chi(-1)=(-1)^j}}\overline{\chi(m_1)}\theta_{j,\chi}\vert V_{g^2}(\tau).
\]
If $ \chi $ is a character modulo $ M_1 $ such that $ \chi(-1) = (-1)^j $ and $ \chi_0 $ is the trivial character, then by \Cref{prop:hol1,prop:hol2},
\begin{align*}
	\pihol\left(\left[\left(\widehat{\calH}\vert V_b\right)^{-},\theta_{j,\chi}\vert V_{g^2}\right]_{\left\lfloor\frac{k}{2}\right\rfloor}\right)& = g^{-j}b^{-\frac{1}{2}}2^{-1-k}\binom{k}{\left\lfloor k/2\right\rfloor}\Lambda_{k+1,g^2,b}^{\chi,\chi_0}\\
	& = g^{-j}b^{-\frac{1}{2}}2^{-k}\binom{k}{\left\lfloor k/2\right\rfloor}\sum_{n=1}^{\infty}\sideset{}{^*}\sum_{\substack{t \geq 1, s \geq 0\\g^2t^2-bs^2 = n}}\chi(t)\left(gt-\sqrt{b}s\right)^{k+1}q^n.
\end{align*}
Moreover, $ \pihol\left(\left[\widehat{\calH}\vert V_b,\theta_{j,\chi}\vert V_{g^2}\right]_{\lfloor\frac{k}{2}\rfloor}\right) $ is a holomorphic cusp form of weight $ k+2 $ if $ k \geq 1 $ and quasimodular of weight $ 2 $ if $ k = 0 $. A short calculation using \eqref{eq:holH} completes the proof.
\end{proof}

Finally, we point out that if $ b $ is a square, then $ \lambda_{\ell,m,M,b}(n) $ has an alternative form, similarly as in \cite[Proposition 3.3]{BKP}.

\begin{lemma}
\label{lem:alt}
Let $ \ell \in \Z_{\geq 0} $, $ m \in \Z $, $ M \in \Z_{\geq 1} $, and $ b = b_1^2 $, where $ b_1 \in \Z_{\geq 1} $. If $ n \in \Z_{\geq 1} $, then
\[
	\lambda_{\ell,m,M,b}(n) = \sum_{\varepsilon\in\{\pm 1\}}\varepsilon^\ell\sideset{}{^*}\sum_{\substack{d|n, d\leq \sqrt{n}\\\frac{n}{d}+d\equiv \varepsilon 2m \pmod{2M}\\\frac{n}{d}-d\equiv 0 \pmod{2b_1}}}d^{\ell+1},
\]
where the $ * $ next to the sum indicates that the term with $ d = \sqrt{n} $ is counted with weight $ \frac{1}{2} $. 
\end{lemma}
\begin{proof}
By definition,
\[
	\lambda_{\ell,m,M,b}(n) := \sum_{\varepsilon\in\{\pm 1\}}\varepsilon^\ell\sideset{}{^*}\sum_{\substack{t \geq 1, s \geq 0\\t^2-b_1^2s^2=n\\t\equiv \varepsilon m \pmod{M}}}\left(t-b_1s\right)^{\ell+1}.
\]
The condition $ t^2-b_1^2s^2 = n $ is equivalent to $ (t-b_1s)(t+b_1s) = n $. If $ t > s \geq 0 $ satisfy this equality and $ t \equiv \varepsilon m \pmod{M} $, then we let $ d := t-b_1s $, and we get $ \frac{n}{d}+d = 2t \equiv \varepsilon 2m \pmod{2M} $ and $ \frac{n}{d}-d = 2b_1s \equiv 0 \pmod{2b_1} $. Conversely, if $ d \mid n $, $ d \leq \sqrt{n} $, is such that $ \frac{n}{d}+d \equiv \varepsilon 2m \pmod{2M} $ and $ \frac{n}{d}-d \equiv 0 \pmod{2b_1} $, then there is a unique solution to the system of equations $ t-b_1s = d $, $ t+b_1s = \frac{n}{d} $. 
\end{proof}

\section{Moments of Hurwitz class numbers}
\label{sec:hur}

In this section, we investigate moments of Hurwitz class numbers. If $ b \in \Z_{\geq 1} $, $ m \in \Z $, and $ M \in \Z_{\geq 1} $, then for $ k \in \Z_{\geq 0} $, the \emph{$ k $-th moment of Hurwitz class numbers} is defined by
\[
	H_{k,m,M,b}(n) := \sum_{t \equiv m \pmod{M}}H\left(\frac{n-t^2}{b}\right)t^k.
\]
We let
\[
	\calH_{k,m,M,b}(\tau) := \sum_{n=0}^{\infty} H_{k,m,M,b}(n)q^n
\]
be the generating function. If $ p $ is prime and $ r \in \Z_{\geq 1} $, then we also let
\[
	\mathscr{H}_{k,m,M,b}(p^r) := \sum_{\substack{p \nmid t\\t \equiv m \pmod{M}}}H\left(\frac{p^r-t^2}{b}\right)t^k.
\]

We note that we are taking these moments with $ H\left(\frac{p^r-t^2}{b}\right) $ instead of $ H\left(\frac{4p^r-t^2}{b}\right) $. This is due to the fact that in \Cref{prop:OSS}, the denominator is either $ 4 $ or $ 16 $. In order for the class number to be non-zero, $ t=2t_1 $ must be even and we can simplify the fraction.

\begin{definition}
\label{def:gamma}
If $ k \in \Z_{\geq 0} $, $ m \in \Z $, $ M \in \Z_{\geq 1} $, $ b \in \Z_{\geq 1} $, and $ j \in \{0,1\} $ is such that $ j \equiv k \pmod{2} $, then we let
\[
	f_{k,m,M,b} := \pihol\left(\left[\widehat{\calH}\vert V_b,\theta_{j,m,M}\right]_{\lfloor\frac{k}{2}\rfloor}\right) = [\calH\vert V_b,\theta_{j,m,M}]_{\left\lfloor\frac{k}{2}\right\rfloor}+b^{-\frac{1}{2}}2^{-1-k}\binom{k}{\left\lfloor k/2\right\rfloor}\Lambda_{k,m,M,b}.
\]
Let $ \gamma_{k,m,M,b} $ be the $ n $-th Fourier coefficient of $ f_{k,m,M,b} $, so that
\[
	f_{k,m,M,b}(\tau) = \sum_{n=0}^{\infty}\gamma_{k,m,M,b}(n)q^n.
\]
\end{definition}
By \Cref{prop:holH}, $ f_{k,m,M,b} $ is a holomorphic cusp form of weight $ k+2 $ on $ \Gamma_{4M^2b,M} = \Gamma_0(4M^2b)\cap\Gamma_1(M) $ (respectively $ \Gamma_0(4M^2b) $) if $ M \nmid m $ (respectively $ M \mid m $) if $ k \geq 1 $ and quasimodular on that group if $ k = 0 $.

Let $ c_{k,m,M,b}(n) $ denote the $ n $-th Fourier coefficient of $ [\calH\vert V_b,\theta_{r,m,M}]_{\left\lfloor\frac{k}{2}\right\rfloor} $, i.e.,
\[
	[\calH\vert V_b,\theta_{m,M}]_{\left\lfloor\frac{k}{2}\right\rfloor} = \sum_{n=0}^{\infty}c_{k,m,M,b}(n)q^n.
\]
We define the numbers
\[
	\alpha_{k,m,M,b}(n) := \binom{k}{\left\lfloor k/2 \right\rfloor}^{-1}\gamma_{k,m,M,b}(n),
\]
so that by \Cref{def:gamma},
\[
	\binom{k}{\left\lfloor k/2 \right\rfloor}^{-1}c_{k,m,M,b}(n) = \alpha_{k,m,M,b}(n)-b^{-\frac{1}{2}}2^{-1-k}\lambda_{k,m,M,b}(n).
\]

\subsection{Recurrence formulas}

First, we use results of Cohen \cite{Co} to prove a recurrence for the moments. The following lemma is a slight extension of \cite[(2.2)]{KPY}.

\begin{lemma}
\label{lem:rec}
If $ k \in \Z_{\geq 0} $, $ m \in \Z $, and $ M \in \Z_{\geq 1} $, then
\[
	\binom{k}{\left\lfloor k/2 \right\rfloor}^{-1}c_{k,m,M,b}(n) = \sum_{\mu=0}^{\left\lfloor k/2 \right\rfloor}(-1)^\mu 2^{-2\mu}\binom{k-\mu}{\mu}n^{\mu}H_{k-2\mu,m,M,b}(n).
\]
\end{lemma}
\begin{proof}
To simplify notation, we let $ \partial_{\tau}^{\mu} := \frac{\partial^{\mu}}{\partial\tau^{\mu}} $. For two holomorphic functions $ f_1 $ and $ f_2 $, $ k_1, k_2 \in \R $, and $ j \in \Z_{\geq 0} $, we have by \cite[Theorem 7.1 b)]{Co}
\[
	[f_2,f_1]_j = \frac{1}{(2\pi i)^j}\sum_{\mu=0}^j (-1)^\mu \frac{\Gamma(j+k_1)\Gamma(2j-\mu+k_1+k_2-1)}{\mu!(j-\mu)!\Gamma(j-\mu+k_1)\Gamma(j+k_1+k_2-1)}\partial_{\tau}^{\mu}\left(f_2 \cdot \partial_{\tau}^{j-\mu}f_1\right).
\]
Let $ \ell \in \{0,1\} $ be such that $ \ell \equiv k \pmod{2} $. For $ f_1 = \theta_{\ell,m,M} $ and $ f_2 = \calH\vert V_b $ we have
\begin{align*}
	\partial_{\tau}^{j-\mu}f_1& = \partial_{\tau}^{j-\mu}\sum_{t\equiv m \pmod{M}}t^{\ell}q^{t^2} = (2\pi i)^{j-\mu}\sum_{t \equiv m \pmod{M}}t^{2j-2\mu+\ell}q^{t^2},\\
	\partial_{\tau}^{\mu}\left(f_2\cdot\partial_{\tau}^{j-\mu}f_1\right)& = (2\pi i)^j\sum_{n=0}^{\infty}\sum_{t \equiv m \pmod{M}}H\left(\frac{n-t^2}{b}\right)t^{2j-2\mu+\ell}n^\mu q^n.
\end{align*}
If $ k_1 = \frac{1}{2}+\ell $ and $ k_2 = \frac{3}{2} $, then we get (multiplying the equality by $ \sqrt{\pi} $)
\begin{multline*}
	\frac{\sqrt{\pi}\Gamma(j+\ell+1)}{\Gamma(j+k_1)}[\calH\vert V_b,\theta_{m,M}]_j = \sum_{n=0}^{\infty}\sum_{t\equiv m \pmod{M}}H\left(\frac{n-t^2}{b}\right)\\
	\cdot\sum_{\mu=0}^j(-1)^{\mu}\frac{\sqrt{\pi}\Gamma(2j-\mu+\ell+1)}{\mu!(j-\mu)!\Gamma(j-\mu+k_1)}t^{2j-2\mu+\ell}n^{\mu}q^n.
\end{multline*}
After we use the duplication formula $ \Gamma(z)\Gamma\left(z+\frac{1}{2}\right) = \sqrt{\pi}\Gamma(2z)2^{1-2z} $ on both sides, this simplifies to
\begin{multline*}
	2^{2j+\ell}\frac{\Gamma(j+\ell+1)\Gamma(j+1)}{\Gamma(2j+\ell+1)}[\calH\vert V_b,\theta_{m,M}]_j = \sum_{n=0}^{\infty}\sum_{t \equiv m \pmod{M}}H\left(\frac{n-t^2}{b}\right)\\
	\cdot\sum_{\mu=0}^j(-1)^{\mu}\frac{\Gamma(2j-\mu+\ell+1)}{\mu!\Gamma(2j-2\mu+\ell+1)}2^{2j-2\mu+\ell}t^{2j-2\mu+\ell}n^\mu q^n
\end{multline*}
(as can be seen by evaluating the expression separately for $ \ell=0$ and $ \ell=1 $). Comparing the coefficients of $ q^n $ completes the proof.
\end{proof}

The next proposition is similar to \cite[(2.7) and Theorem 2.3]{KPY}. It is stated in terms of the numbers $ T(k,\mu) = \frac{k-2\mu+1}{k-\mu+1}\binom{k}{\mu} $ defined in \eqref{eq:T}.

\begin{proposition}
\label{prop:Hk}
Let $ k \in \Z_{\geq 0} $, $ m \in \Z $, $ M \in \Z_{\geq 1} $, and $ b \in \Z_{\geq 1} $. We have
\begin{multline*}
	H_{k,m,M,b}(n) = \delta_{2\mid k}C_{\frac{k}{2}}2^{-k}n^{k/2}H_{0,m,M,b}(n)\\
	+\sum_{\mu = 0}^{\left\lfloor(k-1)/2\right\rfloor}2^{-2\mu}T(k,\mu)\left(\alpha_{k-2\mu,m,M,b}(n)-b^{-\frac{1}{2}}2^{-1-k+2\mu}\lambda_{k-2\mu,m,M,b}(n)\right)n^{\mu}.
\end{multline*}
\end{proposition}
\begin{proof}
To simplify the notation, we let
\[
	A_{\ell,m,M,b}(n) := \alpha_{\ell,m,M,b}(n)-b^{-\frac{1}{2}}2^{-1-\ell}\lambda_{\ell,m,M,b}(n).
\]
By \Cref{lem:rec},
\[
	H_{k,m,M,b}(n) = A_{k,m,M,b}(n)-\sum_{j=1}^{\left\lfloor k/2 \right\rfloor} (-1)^j2^{-2j}\binom{k-j}{j}n^{j}H_{k-2j,m,M,b}(n).
\]
We see from this recurrence that for $ \ell \in \Z_{\geq 1} $ there exist some rational numbers $ U(\ell,\mu) $ such that
\[
	H_{\ell,m,M,b}(n) = \delta_{2\mid\ell}U(\ell,\ell/2)2^{-\ell}n^{\ell/2}H_{0,m,M,b}(n)+\sum_{\mu=0}^{\left\lfloor(\ell-1)/2\right\rfloor}2^{-2\mu}U(\ell,\mu)A_{\ell-2\mu,m,M,b}(n)n^\mu.
\]
After substituting this into the previous formula, we obtain the following relations for the numbers $ U(k,\mu) $:
\[
	U(k,0) = 1,\quad U(k,\mu) = -\sum_{j=1}^{\mu} (-1)^j\binom{k-j}{j}U(k-2j,\mu-j),\quad 1 \leq \mu \leq \left\lfloor\frac{k}{2}\right\rfloor.
\]
This is the same recurrence as in \cite[(2.8)]{KPY}, and these numbers satisfy $ U(k,\mu) = T(k,\mu) $ by \cite[Theorem 2.3]{KPY}. Moreover, $ T(k,k/2) = C_{\frac{k}{2}} $.
\end{proof}

\subsection{Relations between moments}

The purpose of the next lemma is to express the previously defined moments $ \mathscr{H}_{k,m,M,b}(p^r) $ with the added condition $ p \nmid t $ in terms of the moments $ H_{k,m,M,b} $ when $ b $ is a square. It is a modification of \cite[Lemma 4.1]{KP1}.

\begin{lemma}
\label{lem:Hp}
Let $ p \geq 3 $ be a prime, $ r \in \Z_{\geq 1} $, and $ b $ a square of a positive integer such that $ p \nmid b $. Let $ m \in \Z $, $ M \in \Z_{\geq 1} $, and $ k \in \Z_{\geq 0} $.
\begin{enumerate}
\item If $ r = 1 $ or both $ p \mid M $ and $ p \nmid m $, then
\[
	\mathscr{H}_{k,m,M,b}(p^r) = H_{k,m,M,b}(p^r)-\delta_{M\mid m}\delta_{k=0}\delta_{b=1}H(p).
\]
\item If $ r \geq 2 $ and $ p \nmid M $, then
\begin{align*}
	\mathscr{H}_{k,m,M,b}(p^r) = &H_{k,m,M,b}(p^r)-p^{k+1}H_{k,m\bar{p},M,b}(p^{r-2})-\delta_{M\mid m}\delta_{k=0}\delta_{2\nmid r}\delta_{b=1}H(p)\\
	&-\delta_{2\mid r}\frac{1}{12}p^{\frac{rk}{2}}(p-1)\left(\delta_{p^{\frac{r}{2}}\equiv m \pmod{M}}+(-1)^k\delta_{-p^{\frac{r}{2}}\equiv m \pmod{M}}\right),
\end{align*}
where $ \bar{p} $ denotes the multiplicative inverse of $ p $ modulo $ M $.
\end{enumerate}
\end{lemma}
\begin{proof}
We have
\begin{equation}
\label{eq:Hp1}
	\mathscr{H}_{k,m,M,b}(p^r) = H_{k,m,M,b}(p^r)-\sum_{\substack{t \equiv m \pmod{M}\\p\mid t}}H\left(\frac{p^r-t^2}{b}\right)t^k.
\end{equation}
We start by proving (1). If $ r = 1 $, then only the term for $ t = 0 $ can appear in the sum. This term contributes $ -H\left(\frac{p}{b}\right) $ only if $ M \mid m $ and $ k = 0 $. If $ p \mid M $ and $ p \nmid m $, then the sum is empty and $ \delta_{M \mid m} = 0 $.

Next, we prove (2). Assume $ r \geq 2 $. Letting $ t \mapsto pt $, we get
\begin{align*}
	\sum_{\substack{t \equiv m \pmod{M}\\p \mid t}}H\left(\frac{p^r-t^2}{b}\right)t^k& = p^k\sum_{pt \equiv m \pmod{M}}H\left(\frac{p^{r-2}-t^2}{b}p^2\right)t^k\\
	& = p^k\sum_{t \equiv m\bar{p} \pmod{M}}H\left(\frac{p^{r-2}-t^2}{b}p^2\right)t^k.
\end{align*}

Let $ D := \frac{t^2-p^{r-2}}{b} $. If $ D $ is a negative integer, then we set $ D_p := \frac{D}{p^{2\alpha}} $, where $ \alpha \in \Z_{\geq 0} $ is such that $ 2\alpha \leq \ord_p(D) \leq 2\alpha+1 $. First, consider the case $ 0 < |t| < p^{\frac{r-2}{2}} $. Then $ \ord_p(D) = \ord_p(t^2-p^{r-2}) = \ord(t^2) $, hence $ \alpha = \ord_p(t) $ and
\[
	H\left(\frac{p^{r-2}-t^2}{b}p^2\right) = pH\left(\frac{p^{r-2}-t^2}{b}\right)+\left(1-\leg{D_p}{p}\right)H\left(|D_p|\right)
\]
by \Cref{lem:delta1}. We have
\[
	D_p = \frac{t^2-p^{r-2}}{bp^{2\alpha}} = \frac{\frac{t^2}{p^{2\alpha}}-p^{r-2-2\alpha}}{b},
\]
hence $ \leg{D_p}{p} = 1 $, where we used $ 2\alpha < r-2 $ because $ 0<|t|<p^{\frac{r-2}{2}} $ and the fact that $ b $ is a square. Thus,
\begin{equation}
\label{eq:Hp2}
	H\left(\frac{p^{r-2}-t^2}{b}p^2\right) = pH\left(\frac{p^{r-2}-t^2}{b}\right).
\end{equation}
We note that this equality remains valid when $ D $ is not an integer. Secondly, consider the case $ t = 0 $ and $ b = 1 $. Then $ \alpha = \frac{r-2}{2} $ if $ 2 \mid r $ and $ \alpha = \frac{r-3}{2} $ if $ 2 \nmid r $, hence
\[
	H\left(p^r\right) = pH\left(p^{r-2}\right)+\delta_{2\mid r}\left(1-\leg{-1}{p}\right)H(1)+\delta_{2\nmid r}H(p) = pH\left(p^{r-2}\right)+\delta_{2\nmid r}H(p),
\]
where we used $ H(1) = 0 $. We also have
\begin{equation}
\label{eq:Hp3}
	H\left(\frac{p^r}{b}\right) = pH\left(\frac{p^{r-2}}{b}\right)+\delta_{2\nmid r}H\left(\frac{p}{b}\right)
\end{equation}
because both sides are $ 0 $ if $ b > 1 $ due to the assumption $ p \nmid b $.

From \eqref{eq:Hp2} and \eqref{eq:Hp3},
\begin{align*}
	&\sum_{t \equiv m\bar{p} \pmod{M}}H\left(\frac{p^{r-2}-t^2}{b}p^2\right)t^k\\
	= &\delta_{M \mid m}\delta_{k=0}H\left(\frac{p^r}{b}\right)+\sum_{\substack{t\equiv m\bar{p}\pmod{M}\\0<|t|<p^{\frac{r-2}{2}}}}H\left(\frac{p^{r-2}-t^2}{b}p^2\right)t^k+\sum_{\substack{t \equiv m\bar{p} \pmod{M}\\t = \pm p^{\frac{r-2}{2}}}}H(0)t^k\\
	= &\delta_{M \mid m}\delta_{k=0}\left(pH\left(\frac{p^{r-2}}{b}\right)+\delta_{2\nmid r}H\left(\frac{p}{b}\right)\right)+\sum_{\substack{t \equiv m\bar{p} \pmod{M}\\0<|t|<p^{\frac{r-2}{2}}}}pH\left(\frac{p^{r-2}-t^2}{b}\right)t^k\\
	&+\delta_{2\mid r}p^{\frac{(r-2)k}{2}}H(0)\left(\delta_{p^{\frac{r}{2}}\equiv m \pmod{M}}+(-1)^k\delta_{-p^{\frac{r}{2}}\equiv m \pmod{M}}\right)\\
	= &\sum_{t \equiv m\bar{p} \pmod{M}}pH\left(\frac{p^{r-2}-t^2}{b}\right)t^k+\delta_{M \mid m}\delta_{k=0}\delta_{2\nmid r}H\left(\frac{p}{b}\right)\\
	&-\delta_{2\mid r}p^{\frac{(r-2)k}{2}}H(0)(p-1)\left(\delta_{p^{\frac{r}{2}} \equiv m \pmod{M}}+(-1)^k\delta_{-p^{\frac{r}{2}} \equiv m \pmod{M}}\right).
\end{align*}
The formula in (2) now follows directly from \eqref{eq:Hp1} together with $ H(0) = -\frac{1}{12} $.
\end{proof}

In order to compute the moments in the Legendre family, we will need to evaluate the even moments of Hurwitz class numbers for $ M = 2 $ and the odd moments for $ M = 4 $. Some useful properties of the coefficients $ \lambda_{\ell,m,M,b}(n) $ for these cases are collected in the next two lemmas.

\begin{lemma}
\label{lem:lambda}
Let $ p \geq 3 $ be a prime and $ r \in \Z_{\geq 1} $.
\begin{enumerate}
\item If $ \ell \in \Z_{\geq 0} $ is even and $ m \in \{0,1\} $, then
\[
	\lambda_{\ell,m,2,1}(p^r)-p^{\ell+1}\lambda_{\ell,m,2,1}\left(p^{r-2}\right) = 2\delta_{p^r \equiv 2m-1 \pmod{4}}.
\]
Moreover, $ \lambda_{\ell,m,2,1}(n) = 0 $ for every $ n \equiv 2m+1 \pmod{4} $.
\item If $ \ell \in \Z_{\geq 1} $ is odd, $ m \in \{1,3\} $, and $ b \in \{1,4\} $, then
\[
	\lambda_{\ell,m,4,b}(p^r)-p^{\ell+1}\lambda_{\ell,\bar{p}m,4,b}\left(p^{r-2}\right) = \delta_{p^r \equiv 1 \pmod{4}}\leg{-1}{m}\leg{2}{p^r},
\]
where $ \bar{p} $ denotes the multiplicative inverse of $ p $ modulo $ 4 $.
\end{enumerate}
The terms $ \lambda_{\ell,m,2,1}\left(p^{r-2}\right) $ and $ \lambda_{\ell,\bar{p}m,4,b}\left(p^{r-2}\right) $ are interpreted as $ 0 $ for $ r = 1 $.
\end{lemma}
\begin{proof}
If $ b = b_1^2 $, then by \Cref{lem:alt},
\[
	\lambda_{\ell,m,M,b}(p^r) = \sum_{\varepsilon\in\{\pm 1\}}\varepsilon^\ell\sideset{}{^*}\sum_{\substack{d|p^r, d\leq \sqrt{p^r}\\\frac{p^r}{d}+d\equiv \varepsilon 2m \pmod{2M}\\\frac{p^r}{d}-d\equiv 0 \pmod{2b_1}}}d^{\ell+1}.
\]

%%%%%%%%%%%%%%%%%%%%%%%%%%%%%%%%%%%%%%%%%%%%%%%%%%%%%%%%%%%%%%%%%%%%%
%To prove (1), we note that if $ p^r \equiv 2m+1 \pmod{4} $, then $ \frac{p^r}{d}+d \equiv (2m+2)d \pmod{4} $ for $ d \mid p^r $, hence the condition $ \frac{p^r}{d}+d \equiv \varepsilon 2m \pmod{4} $ is equivalent to $ (m+1)d \equiv m \pmod{2} $, which is never satisfied. On the other hand, if $ p^r \equiv 2m-1 \pmod{4} $, then $ \frac{p^r}{d}+d \equiv 2md \pmod{4} $, hence $ \frac{p^r}{d}+d \equiv \varepsilon 2m \pmod{4} $ is always satisfied. Since $ \ell $ is even, we get
%\[
%	\lambda_{\ell,m,2,1}(p^r)-p^{\ell+1}\lambda_{\ell,m,2,1}\left(p^{r-2}\right) = 2\left(\sideset{}{^*}\sum_{d|p^r, d\leq \sqrt{p^r}}d^{\ell+1}-\sideset{}{^*}\sum_{d|p^{r-2}, d\leq \sqrt{p^{r-1}}}d^{\ell+1}\right) = 2.
%\]
%%%%%%%%%%%%%%%%%%%%%%%%%%%%%%%%%%%%%%%%%%%%%%%%%%%%%%%%%%%%%%%%%%%%%

Let us prove (2). If $ p^r \equiv 3 \pmod{4} $, then $ \frac{p^r}{d}+d \equiv 0 \pmod{4} $ but $ \varepsilon 2m \equiv 2 \pmod{4} $. It follows that $ \lambda_{\ell,m,4,b}(p^r) = \lambda_{\ell,m,4,b}\left(p^{r-2}\right) = 0 $. Next, we assume $ p^r \equiv 1 \pmod{4} $. Since $ b_1 \in \{1,2\} $, the condition $ \frac{p^r}{d}-d \equiv 0 \pmod{2b_1} $ is automatically satisfied. If $ p^r \equiv 1 \pmod{8} $, then
\[
	\lambda_{\ell,m,4,b}(p^r)-\lambda_{\ell,\bar{p}m,4,b}\left(p^{r-2}\right) = \sum_{\varepsilon\in\{\pm 1\}}\varepsilon\left(\sideset{}{^*}\sum_{\substack{d|p^r, d\leq \sqrt{p^r}\\d\equiv \varepsilon m \pmod{4}}}d^{\ell+1}-\sideset{}{^*}\sum_{\substack{d|p^{r-2}, d\leq \sqrt{p^{r-2}}\\d\equiv \varepsilon \bar{p}m \pmod{4}}}d^{\ell+1}\right).
\]
Letting $ pd \mapsto d $ in the second inner sum, we get
\[
	\lambda_{\ell,m,4,b}(p^r)-\lambda_{\ell,\bar{p}m,4,b}\left(p^{r-2}\right) = \sum_{\varepsilon\in\{\pm 1\}}\varepsilon \delta_{1 \equiv \varepsilon m \pmod{4}} = \leg{-1}{m}.
\]
The case $ p^r \equiv 5 \pmod{8} $ is analogous.
\end{proof}

\subsection{Zeroth moment}

Let us review some basic facts about quasimodular forms. It is well known that the generating function
\[
	\calD(\tau) := \sum_{n=1}^{\infty}\sigma(n)q^n,
\]
where $ \sigma(n) := \sum_{d\mid n}d $, is a quasimodular form of weight $ 2 $ on $ \Sl_2(\Z) $. In fact, if
\[
	\widehat{\calD}(\tau) := \calD(\tau)-\frac{1}{24}+\frac{1}{8\pi v},
\]
where $ v = \Im(\tau) $, then $ \widehat{\calD}(\tau) $ satisfies weight $ 2 $ modularity on $ \Sl_2(\Z) $, see \cite[p. 113]{Ko}.

\Cref{lem:S} captures how the sieving operator $ S_{M,m} $ changes the level of a holomorphic modular form. We point out that the lemma can be easily extended to quasimodular forms because the space of quasimodular forms of weight $ 2 $ on $ \Gamma_{N_0,N_1} $ is spanned by $ \calD $ together with a basis of $ M_2(\Gamma_{N_0,N_1}) $. In particular, if $ M \mid 24 $ and $ M \nmid m $, then $ \calD\vert S_{M,m} \in M_2(\Gamma_0(M^2)) $.

Next, we explicitly compute the sum of the Hurwitz class numbers, i.e., the zeroth moment, for $ M = 2 $.

\begin{lemma}
\label{lem:H0}
Let $ m \in \{0,1\} $. If $ n \in \Z_{\geq 1} $ is odd, then
\[
	H_{0,m,2,1}(n) = \begin{cases}
		\frac{1}{3}\sigma(n)-\frac{1}{2}\lambda_{0,m,2,1}(n)&\text{if $ n \equiv 2m-1 \pmod{4} $},\\
		0&\text{if $ n \equiv 2m+1 \pmod{4} $.}
	\end{cases}
\]
\end{lemma}
\begin{proof}
Let $ \chi_{-4} $ denote the non-principal character modulo $ 4 $. If we show that
\begin{equation}
\label{eq:H0}
	\left(\calH_{0,m,2,1}+\frac{1}{2}\Lambda_{0,m,2,1}\right)\otimes\chi_{-4}^2 = \frac{1}{6}\calD\otimes\chi_{-4}^2+(-1)^{m+1}\frac{1}{6}\calD\otimes\chi_{-4},
\end{equation}
then the lemma immediately follows from the fact that $ \lambda_{\ell,m,2,1}(n) = 0 $ if $ n \equiv 2m+1 \pmod{4} $ by \Cref{lem:lambda}.

We recall that the Rankin-Cohen bracket $ [\calH, \theta_{m,M}]_0 $ is simply the product $ \calH\theta_{m,M} $. If $ m \in \{0, 1\} $, then by \Cref{prop:holH},
\[
	\pihol\left(\widehat{\calH}\theta_{m,2}\right) = \calH\theta_{m,2}+\frac{1}{2}\Lambda_{0,m,2,1} = \calH_{0,m,2,1}+\frac{1}{2}\Lambda_{0,m,2,1}
\]
is a quasimodular form of weight $ 2 $ on $ \Gamma_0(16)\cap\Gamma_1(2) = \Gamma_0(16) $. Both the left-hand side and the right-hand side of \eqref{eq:H0} are modular forms in $ M_2(\Gamma_0(16)) $ and we have $ [\Sl_2(\Z):\Gamma_0(16)] = 24 $ (a formula for the index can be found, e.g., in \cite[Lemma 2.2]{BK2}). It follows from the Sturm bound that \eqref{eq:H0} holds if the coefficients of $ q^n $ on both sides agree for $ 0 \leq n \leq 4 $, which can be easily checked.
\end{proof}

\subsection{Higher moments}

In the next two theorems, we find relatively simple expressions for the moments with the added condition $ p \nmid t $ in the case $ M = 2 $ and $ M = 4 $.

\begin{theorem}
\label{thm:Hp2}
Let $ k \in \Z_{\geq 1} $ be even, $ m \in \{0,1\} $, $ p \geq 3 $ be a prime and $ r \in \Z_{\geq 1} $. If $ p^r \equiv 2m-1 \pmod{4} $, then
\begin{multline*}
	\mathscr{H}_{k,m,2,1}(p^r) = \frac{C_{\frac{k}{2}}}{3\cdot 2^k}\left(p^{r\left(\frac{k}{2}+1\right)}-2p^{\frac{rk}{2}}\right)-\delta_{2 \mid r}\frac{1}{6}p^{\frac{rk}{2}}(p-1)-2^{-k}\sum_{\mu=0}^{\frac{k}{2}-1}T(k,\mu)p^{r\mu}\\
	+\sum_{\mu=0}^{\frac{k}{2}-1}2^{-2\mu}T(k,\mu)\left(\alpha_{k-2\mu,m,2,1}(p^r)-p^{k-2\mu+1}\alpha_{k-2\mu,m,2,1}(p^{r-2})\right)p^{r\mu},
\end{multline*}
where we take $ \alpha_{k-2\mu,m,2,1}(p^{r-2}) = 0 $ in the case $ r = 1 $. If $ p^r \equiv 2m+1 \pmod{4} $, then $ \mathscr{H}_{k,m,2,1}(p^r) = 0 $.
\end{theorem}
\begin{proof}
We will prove this for $ r \geq 2 $. The proof for $ r = 1 $ is obtained by interpreting $ H_{k,m,2,1}(p^{-1}) $ and $ \lambda_{\ell,m,2,1}(p^{-1}) $ in the following calculation as zero. By \Cref{lem:Hp} with $ M = 2 $, we have
\begin{multline}
\label{eq:HpM2}
	\mathscr{H}_{k,m,2,1}(p^r) = H_{k,m,2,1}(p^r)-p^{k+1}H_{k,m,2,1}(p^{r-2})\\
	-\delta_{2 \mid r}\frac{1}{12}p^{\frac{rk}{2}}(p-1)\left(\delta_{p^{\frac{r}{2}}\equiv m \pmod{2}}+\delta_{p^{\frac{r}{2}}\equiv m \pmod{2}}\right).
\end{multline}

First, assume that $ p^r \equiv 2m-1 \pmod{4} $. The last term in \eqref{eq:HpM2} simplifies to
\[
	-\delta_{2 \mid r}\frac{1}{6}p^{\frac{rk}{2}}(p-1).
\]
We use the formula
\begin{multline*}
	H_{k,m,2,1}(n) = C_{\frac{k}{2}}2^{-k}n^{k/2}H_{0,m,2,1}(n)\\
	+\sum_{\mu = 0}^{\frac{k}{2}-1}2^{-2\mu}T(k,\mu)\left(\alpha_{k-2\mu,m,2,1}(n)-2^{2\mu-k-1}\lambda_{k-2\mu,m,2,1}(n)\right)n^{\mu}
\end{multline*}
from \Cref{prop:Hk}. By \Cref{lem:lambda}, we have
\[
	\lambda_{k-2\mu,m,2,1}(p^r)-p^{k-2\mu+1}\lambda_{k-2\mu,m,2,1}(p^{r-2}) = 2,
\]
and by \Cref{lem:H0}, we have
\[
	H_{0,m,2,1}(n) = \frac{1}{3}\sigma(n)-\frac{1}{2}\lambda_{0,m,2,1}(n)
\]
if $ n \equiv 2m-1 \pmod{4} $, hence
\[
	H_{0,m,2,1}(p^r)-pH_{0,m,2,1}(p^{r-2}) = \frac{1}{3}(p^r-2).
\]
Thus,
\begin{align*}
	&\mathscr{H}_{k,m,2,1}(p^r) = C_{\frac{k}{2}}2^{-k}p^{\frac{rk}{2}}\left(H_{0,m,2,1}(p^r)-pH_{0,m,2,1}(p^{r-2})\right)-\delta_{2\mid r}\frac{1}{6}p^{\frac{rk}{2}}(p-1)\\
	&+\sum_{\mu=0}^{\frac{k}{2}-1}2^{-2\mu}T(k,\mu)\left(\alpha_{k-2\mu,m,2,1}(p^r)-p^{k-2\mu+1}\alpha_{k-2\mu,m,2,1}(p^{r-2})\right)p^{r\mu}\\
	&-2^{-k-1}\sum_{\mu=0}^{\frac{k}{2}-1}T(k,\mu)\left(\lambda_{k-2\mu,m,2,1}(p^r)-p^{k-2\mu+1}\lambda_{k-2\mu,m,2,1}(p^{r-2})\right)p^{r\mu}
\end{align*}
and after using the above relations, we obtain the expression for $ \mathscr{H}_{k,m,2,1}(p^r) $ in the statement.

Next, assume that $ p^r \equiv 2m+1 \pmod{4} $. By definition,
\[
	H_{k,m,2,1}(n) = \sum_{t \equiv m \pmod{4}}H(n-t^2)t^k.
\]
If $ n \equiv 2m+1 \pmod{4} $, then $ n-t^2 \equiv 2m+1-m^2 \equiv 1,2 \pmod{4} $, hence $ H(n-t^2) = 0 $. The last term in \eqref{eq:HpM2} also equals zero, and it follows that $ \mathscr{H}_{k,m,2,1}(p^r) = 0 $.
\end{proof}

\begin{theorem}
\label{thm:Hp4}
Let $ k \in \Z_{\geq 1} $ be odd, $ m \in \{1,3\} $, $ b \in \{1,4\} $, $ p \geq 3 $ be a prime and $ r \in \Z_{\geq 1} $. If $ p^r \equiv 1 \pmod{4} $, then
\begin{multline*}
	\mathscr{H}_{k,m,4,b}(p^r) = \sum_{\mu=0}^{(k-1)/2}2^{-2\mu}T(k,\mu)\left(\alpha_{k-2\mu,m,4,b}(p^r)-p^{k-2\mu+1}\alpha_{k-2\mu,m\bar{p},4,b}(p^{r-2})\right)p^{r\mu}\\
	-b^{-\frac{1}{2}}\leg{-1}{m}\leg{2}{p^r}2^{-k-1}\sum_{\mu=0}^{(k-1)/2} T(k,\mu)p^{r\mu}-\delta_{2 \mid r}\frac{1}{12}p^{\frac{rk}{2}}(p-1)\leg{-1}{m}\varepsilon(p^r),
\end{multline*}
where $ \varepsilon(p^r) = \leg{-1}{p^{r/2}} $ and where we take $ \alpha_{k-2\mu,m,4,b}(p^{-1}) = 0 $ in the case $ r = 1 $. If $ p^r \equiv 3 \pmod{4} $, then $ \mathscr{H}_{k,m,4,b}(p^r) = 0 $.
\end{theorem}
\begin{proof}
By \Cref{lem:Hp} with $ M = 4 $, we have
\begin{multline*}
	\mathscr{H}_{k,m,4,b}(p^r) = H_{k,m,4,b}(p^r)-p^{k+1}H_{k,m\bar{p},4,b}(p^{r-2})\\
	-\delta_{2 \mid r}\frac{1}{12}p^{\frac{rk}{2}}(p-1)\left(\delta_{p^{\frac{r}{2}}\equiv m \pmod{4}}-\delta_{-p^{\frac{r}{2}}\equiv m \pmod{4}}\right).
\end{multline*}
First, assume that $ p^r \equiv 1 \pmod{4} $. We note that
\[
	\delta_{p^{\frac{r}{2}}\equiv m \pmod{4}}-\delta_{-p^{\frac{r}{2}} \equiv m \pmod{4}}\ = \leg{-1}{m}\varepsilon(p^r).
\]
By \Cref{lem:lambda},
\[
	\lambda_{k-2\mu,m,4,b}(p^r)-p^{k-2\mu+1}\lambda_{k-2\mu,m\bar{p},4,b}(p^{r-2}) = \leg{-1}{m}\leg{2}{p^r}.
\]
By \Cref{prop:Hk}, we have
\[
	H_{k,m,4,b}(n) = \sum_{\mu=0}^{\left\lfloor (k-1)/2 \right\rfloor}2^{-2\mu}T(k,\mu)\left(\alpha_{k-2\mu,m,4,b}(n)-b^{-\frac{1}{2}}2^{2\mu-k-1}\lambda_{k-2\mu,m,4,b}(n)\right)n^{\mu},
\]
hence
\begin{align*}
	&\mathscr{H}_{k,m,4,b}(p^r) = \sum_{\mu=0}^{(k-1)/2}2^{-2\mu}T(k,\mu)\left(\alpha_{k-2\mu,m,4,b}(p^r)-p^{k-2\mu+1}\alpha_{k-2\mu,m\bar{p},4,b}(p^{r-2})\right)p^{r\mu}\\
	&-b^{-\frac{1}{2}}2^{-k-1}\sum_{\mu=0}^{(k-1)/2} T(k,\mu)\left(\lambda_{k-2\mu,m,4,b}(p^r)-p^{k-2\mu+1}\lambda_{k-2\mu,m\bar{p},4,b}(p^{r-2})\right)p^{r\mu}\\
	&-\delta_{2 \mid r}\frac{1}{12}p^{\frac{rk}{2}}(p-1)\leg{-1}{m}\varepsilon(p^r).
\end{align*}
Using the relation above, the formula follows.

Next, assume that $ p^r \equiv 3 \pmod{4} $, in particular $ r $ is odd. By definition,
\[
	H_{k,m,4,b}(n) = \sum_{t \equiv m \pmod{4}}H\left(\frac{n-t^2}{b}\right)t^k.
\]
If $ n \equiv 3 \pmod{4} $, then $ n-t^2 \equiv 2 \pmod{4} $, hence $ H\left(\frac{n-t^2}{b}\right) = 0 $, and it follows that $ \mathscr{H}_{k,m,4,b}(p^r) = 0 $.
\end{proof}

\section{Moments of the Legendre family}
\label{sec:S}

\subsection{Proof of the main theorem}

Our main result (\Cref{thm:Main}) was proved in the introduction as a consequence of \Cref{thm:Seven,thm:Sodd} below.

\begin{theorem}
\label{thm:Seven}
Let $ p \geq 5 $ be a prime and $ r \in \Z_{\geq 1} $. Let $ m \equiv \frac{p^r+1}{2} \pmod{2} $. If $ k \in \Z_{\geq 1} $ is even, then
\begin{multline*}
	S_k^{\Leg}(p^r) = C_{\frac{k}{2}}\left(p^{r\left(\frac{k}{2}+1\right)}-2p^{\frac{rk}{2}}\right)-3\sum_{\mu=1}^{\frac{k}{2}-1}T(k,\mu)p^{r\mu}-1\\
	+3\sum_{\mu=0}^{\frac{k}{2}-1}2^{k-2\mu}T(k,\mu)\left(\alpha_{k-2\mu,m,2,1}(p^r)-p^{k-2\mu+1}\alpha_{k-2\mu,m,2,1}(p^{r-2})\right)p^{r\mu},
\end{multline*}
where we take $ \alpha_{k-2\mu,m,2,1}(p^{-1}) = 0 $ in the case $ r = 1 $.
\end{theorem}
\begin{proof}
After substituting $ t = 2t_1 $ in the sum, the formula for $ S_k^{\Leg}(p^r) $ from \Cref{prop:OSS} becomes
\[
	S_k^{\Leg}(p^r) = 2+\delta_{2 \mid r}2^{k-1}p^{\frac{rk}{2}}(p-1)+3\cdot 2^k\mathscr{H}_{k,m,2,1}(p^r),
\]
and using the expression for $ \mathscr{H}_{k,m,2,1}(p^r) $ from \Cref{thm:Hp2} completes the proof.
\end{proof}

\begin{theorem}
\label{thm:Sodd}
Let $ p \geq 5 $ be a prime and $ r \in \Z_{\geq 1} $. Assume that $ p^r \equiv 1 \pmod{4} $ and let $ m \equiv \frac{p^r+1}{2} \pmod{4} $. If $ k \in \Z_{\geq 1} $ is odd, then
\begin{align*}
	&S_k^{\Leg}(p^r) = -2\sum_{\mu=1}^{(k-1)/2} T(k,\mu)p^{r\mu}\\
	&+\sum_{\mu=0}^{(k-1)/2}2^{k-2\mu+1}T(k,\mu)\left(\alpha_{k-2\mu,m,4,1}(p^r)-p^{k-2\mu+1}\alpha_{k-2\mu,m\bar{p},4,1}(p^{r-2})\right)p^{r\mu}\\
	&+\sum_{\mu=0}^{(k-1)/2}2^{k-2\mu+2}T(k,\mu)\left(\alpha_{k-2\mu,m,4,4}(p^r)-p^{k-2\mu+1}\alpha_{k-2\mu,m\bar{p},4,4}(p^{r-2})\right)p^{r\mu},
\end{align*}
where we take $ \alpha_{k-2\mu,m\bar{p},4,b}(p^{-1}) = 0 $ for $ b \in \{1,4\} $ in the case $ r = 1 $.
\end{theorem}
\begin{proof}
Again taking the formula from \Cref{prop:OSS} and substituting $ t = 2t_1 $ in the two sums, we have
\[
	S_k^{\Leg}(p^r) = 2+\delta_{2 \mid r}\varepsilon(p^r)2^{k-1} p^{\frac{rk}{2}}(p-1)+2^{k+1}\mathscr{H}_{k,m,4,1}(p^r)+2^{k+2}\mathscr{H}_{k,m,4,4}(p^r).
\]
We plug in the expression for $ \mathscr{H}_{k,m,4,b}(p^r) $ with $ b \in \{1,4\} $ from \Cref{thm:Hp4}. If $ 2 \mid r $, then $ p^r \equiv 1 \pmod{8} $, hence $ m \equiv \frac{p^r+1}{2} \equiv 1 \pmod{4} $ and $ \leg{-1}{m} = 1 $, so the terms with $ \delta_{2 \mid r} $ cancel. Moreover, since $ m \equiv \frac{p^r+1}{2} \pmod{4} $, we always have $ \leg{-1}{m}\leg{2}{p^r} = 1 $. The theorem follows.
\end{proof}

Next, we prove \Cref{cor:Bias} about the averages of the terms in the asymptotic expansion of the moments.
\begin{corollary}
\label{cor:Bias'}
Let $ p \geq 5 $ be a prime, $ r \in \Z_{\geq 1} $, and $ k \in \Z_{\geq 1} $.
\begin{enumerate}
\item If $ k $ is even, then the largest $ j < \frac{k}{2}+1 $ such that $ \calA_{k,r,j}^{\Leg} $ does not vanish is $ j = \frac{k}{2} $. Moreover, if $ j \in \Z $, then
\[
	\calA_{k,r,j}^{\Leg} = \begin{cases}
		-2C_{\frac{k}{2}}&\text{if $ j = \frac{k}{2} $,}\\
		-3T(k,j)&\text{if $ 1 \leq j \leq \frac{k}{2}-1 $,}\\
		-1&\text{if $ j = 0 $,}
	\end{cases}
\]
and if $ j \in \Z+\frac{1}{2} $, then $ \calA_{k,r,j}^{\Leg} = 0 $.
\item If $ k $ is odd, then the largest $ j $ such that $ \calA_{k,r,j}^{\Leg} $ does not vanish is $ j = \frac{k-1}{2} $. If $ j \in \Z $, then
\[
	\calA_{k,r,j}^{\Leg} = \begin{cases}
	-2T(k,j)&\text{if $ 1 \leq j \leq \frac{k-1}{2} $, $ r $ even,}\\
	-T(k,j)&\text{if $ 1 \leq j \leq \frac{k-1}{2} $, $ r $ odd,}\\
	0&\text{if $ j = 0 $,}
	\end{cases}
\]
and if $ j \in \Z+\frac{1}{2} $, then $ \calA_{k,r,j}^{\Leg} = 0 $.
\end{enumerate}
\end{corollary}
\begin{proof}
If $ \ell \in \Z_{\geq 0} $, $ m \in \Z $, $ M \in \Z_{\geq 1} $, and $ b \in \Z_{\geq 1} $, then by definition, $ \alpha_{\ell,m,M,b}(n) = \binom{\ell}{\lfloor \ell/2 \rfloor}^{-1}\gamma_{\ell,m,M,b}(n) $ and $ \gamma_{\ell,m,M,b}(n) $ is a Fourier coefficient of $ f_{\ell,m,M,b} $, which is a cusp form of weight $ \ell+2 $ if $ \ell \geq 1 $. The coefficient $ \alpha_{\ell,m,M,b}(p^r) $ is of size $ p^{r\frac{\ell+1}{2}} $ (here ``size'' does not mean absolute value; instead, it is used in the sense defined in the Introduction).

First, let $ k $ be even. By \Cref{thm:Seven}, we have
\[
	E_k(p^r) = 3\sum_{\mu=0}^{\frac{k}{2}-1}2^{k-2\mu}T(k,\mu)\left(\alpha_{k-2\mu,m,2,1}(p^r)-p^{k-2\mu+1}\alpha_{k-2\mu,m,2,1}(p^{r-2})\right)p^{r\mu}.
\]
The terms in $ E_k(p^r) $ are of size $ p^{r\frac{k+1}{2}} $ because $ \alpha_{k-2\mu,m,2,1}(p^r)p^{r\mu} $ is of size $ p^{r\frac{k-2\mu+1}{2}+r\mu} = p^{r\frac{k+1}{2}} $ and $ p^{k-2\mu+1}\alpha_{k-2\mu,m,2,1}(p^{r-2})p^{r\mu} $ is also of size $ p^{k-2\mu+1+(r-2)\frac{k-2\mu+1}{2}+r\mu} = p^{r\frac{k+1}{2}} $. Hence,
\[
	S_{k,\frac{k+1}{2}}^{\Leg}(p^r) = \frac{E_k(p^r)}{p^{r\frac{k+1}{2}}}.
\]
Moreover, by the Sato-Tate conjecture for Fourier coefficients of cusp forms \cite{CHT}, the normalized coefficients are zero on average:
\[
	\calA_{k,r,\frac{k+1}{2}}^{\Leg} = \lim_{x\to\infty}\frac{1}{\pi(x)}\sum_{p \leq x}S_{k,\frac{k+1}{2}}^{\Leg}(p^r) = 0.
\]
If $ j \in \Z+\frac{1}{2} $ is different from $ \frac{k+1}{2} $, then there are no terms of size $ p^{rj} $, hence $ S_{k,j}^{\Leg}(p^r) = 0 $ and $ \calA_{k,r,j}^{\Leg} = 0 $. If $ j \in \Z $, then we get $ \calA_{k,r,j}^{\Leg} $ directly from \Cref{thm:Seven}.

Secondly, let $ k $ be odd. For $ p^r \equiv 1 \pmod{4} $, we have from \Cref{thm:Sodd}
\begin{align*}
	E_k(p^r)&=\sum_{\mu=0}^{(k-1)/2}2^{k-2\mu+1}T(k,\mu)\left(\alpha_{k-2\mu,m,4,1}(p^r)-p^{k-2\mu+1}\alpha_{k-2\mu,m\bar{p},4,1}(p^{r-2})\right)p^{r\mu}\\
	&+\sum_{\mu=0}^{(k-1)/2}2^{k-2\mu+2}T(k,\mu)\left(\alpha_{k-2\mu,m,4,4}(p^r)-p^{k-2\mu+1}\alpha_{k-2\mu,m\bar{p},4,4}(p^{r-2})\right)p^{r\mu}.
\end{align*}
As before, $ E_k(p^r) $ contributes only terms of size $ p^{r\frac{k+1}{2}} $ and $ \calA_{k,r,\frac{k+1}{2}}^{\Leg} = 0 $. If $ j \in \Z $, $ 1 \leq j \leq \frac{k-1}{2} $, then
\[
	S_{k,j}^{\Leg}(p^r) = -2T(k,\mu)\delta_{p^r \equiv 1 \pmod{4}},
\]
hence
\[
	\calA_{k,r,j}^{\Leg} = \lim_{x\to\infty}\frac{1}{\pi(x)}\sum_{p\leq x}S_{k,j}^{\Leg}(p^r) = \begin{cases}
		-2T(k,\mu)&\text{$ r $ even,}\\
		-T(k,\mu)&\text{$ r $ odd,}
	\end{cases}
\]
where we used that primes congruent to $ 1 \pmod{4} $ have density $ \frac{1}{2} $. If $ j = 0 $ and $ j \in \Z+\frac{1}{2} $, then $ S_{k,j}^{\Leg}(p^r) = 0 $, hence $ \calA_{k,r,j}^{\Leg} = 0 $.
\end{proof}

\subsection{Examples}

We explicitly compute the third and fourth moment in the Legendre family. For that purpose, we work with some specific newforms identified by their LMFDB labels~\cite{LMFDB}.

Let $ g_1 \in S_3(\Gamma_0(64),\chi_{-4}) $ be the newform with LMFDB label 64.3.c.a:
\begin{equation}
	\label{eq:g1}
	g_1(\tau) = \sum_{n=1}^{\infty}c_{g_1}(n)q^n = \frac{\eta(8z)^{18}}{\eta(4z)^6\eta(16z)^6} = q+6q^5+9q^9-10q^{13}+O(q^{17}).
\end{equation}
Let $ g_2, g_3 \in S_5(\Gamma_0(64),\chi_{-4}) $ be the cusp forms with LMFDB labels 64.5.c.a and 64.5.c.c, respectively, so that
\begin{align}
\label{eq:g2}
	g_2(\tau)& = \sum_{n=1}^{\infty}c_{g_2}(n)q^n = \frac{\eta(8z)^{38}}{\eta(4z)^{14}\eta(16z)^{14}} = q+14q^5+81q^9+238q^{13}+O(q^{17}),\\
	g_3(\tau)& = \sum_{n=1}^{\infty}c_{g_3}(n)q^n = q-\beta q^3-18q^5-2\beta q^7-111 q^9+O(q^{11}),\qquad \beta = 8\sqrt{-3}.
\end{align}
We also let $ \Tr(g_3) $ be the trace form of $ g_3 $ given by
\[
	\Tr(g_3)(\tau) = \sum_{n=1}^{\infty}c_{\Tr(g_3)}(n)q^n = 2q-36q^5-222q^9-356q^{13}+O(q^{17}).
\]
Let $ g_4 \in S_6(\Gamma_0(4)) $ be the newform with LMFDB label 4.6.a.a and $ g_5 \in S_6(\Gamma_0(16)) $ the newform with LMFDB label 16.6.a.b:
\begin{align}
\label{eq:g4}
	g_4(\tau)& = \sum_{n=1}^{\infty} c_{g_4}(n)q^n = \eta(2\tau)^{12} = q-12q^3+54q^5-88q^7-99q^9+O\left(q^{10}\right),\\
\label{eq:g5}
	g_5(\tau)& = \sum_{n=1}^{\infty} c_{g_5}(n)q^n = \frac{\eta(4\tau)^{36}}{\eta(2\tau)^{12}\eta(8\tau)^{12}} = q+12q^3+54q^5+88q^7-99q^9+O\left(q^{10}\right).
\end{align}

\begin{proposition}
\label{prop:S3}
If $ p \geq 5 $ be prime and $ r \in \Z_{\geq 1} $, then
\[
	S_3^{\Leg}(p^r) = \begin{cases}
		-4p^r+E_3(p^r)&\text{if $ p^r \equiv 1 \pmod{4} $,}\\
		0&\text{if $ p^r \equiv 3 \pmod{4} $,}
	\end{cases}
\]
where
\[
	E_3(p^r) = -\leg{2}{p^r}\left(c_{g_2}(p^r)-\leg{-1}{p}p^4c_{g_2}(p^{r-2})\right).
\]
\end{proposition}
\begin{proof}
If $ p^r \equiv 3 \pmod{4} $, then the conclusion follows directly from \Cref{prop:OSS}.

Assume that $ p^r \equiv 1 \pmod{4} $ and $ m \equiv \frac{p^r+1}{2} \pmod{4} $. By \Cref{thm:Sodd},
\begin{align*}
	S_3^{\Leg}(p^r)& = -2T(3,1)p^r\\
	&+16T(3,0)\left(\alpha_{3,m,4,1}(p^r)+2\alpha_{3,m,4,4}(p^r)-p^4\alpha_{3,m\bar{p},4,1}(p^{r-2})-2p^4\alpha_{3,m\bar{p},4,4}(p^{r-2}\right)\\
	&+4T(3,1)\left(\alpha_{1,m,4,1}(p^r)+2\alpha_{1,m,4,4}(p^r)-p^2\alpha_{1,m\bar{p},4,1}(p^{r-2})-2p^2\alpha_{1,m\bar{p},4,4}(p^{r-2})\right)p^r.
\end{align*}
The function $ f_{1,m,4,b} $ introduced in \Cref{def:gamma} is a holomorphic cusp form of weight $ 3 $ on $ \Gamma_{64b,4} = \Gamma_0(64b)\cap\Gamma_1(4) $. We claim that
\begin{align*}
	f_{1,m,4,1}\vert S_{2,1}(\tau)& = \frac{(-1)^{\frac{m-1}{2}}}{24}g_1(\tau),\\
	f_{1,m,4,4}\vert S_{2,1}(\tau)& = -\frac{(-1)^{\frac{m-1}{2}}}{48}g_1(\tau).
\end{align*}
To prove the claim, we note that by \Cref{lem:S}, $ f_{1,m,4,b}|S_{2,1} $ is also a holomorphic cusp form of weight $ 3 $ on the same group. The index is $ [\Sl_2(\Z):\Gamma_{64b,4}] = 192b $ and the Sturm bound shows that that we only have to check the coefficients of $ q^n $ for $ 0 \leq n \leq 48b $.

From the claim, we get
\begin{align*}
	\alpha_{1,m,4,1}(p^r)& = \gamma_{1,m,4,1}(p^r) = \frac{(-1)^{\frac{m-1}{2}}}{24}c_{g_1}(p^r),\\
	\alpha_{1,m,4,4}(p^r)& = \gamma_{1,m,4,4}(p^r) = -\frac{(-1)^{\frac{m-1}{2}}}{48}c_{g_1}(p^r),
\end{align*}
hence
\[
	\alpha_{1,m,4,1}(p^r)+2\alpha_{1,m,4,4}(p^r)-p^2\alpha_{1,m\bar{p},4,1}(p^{r-2})-2p^2\alpha_{1,m\bar{p},4,4}(p^{r-2}) = 0.
\]
Secondly, we claim that
\begin{align*}
	f_{3,m,4,1}\vert S_{2,1}(\tau)& = (-1)^{\frac{m+1}{2}}\left(\frac{3}{32}g_2(\tau)-\frac{1}{32}\Tr(g_3)(\tau)\right),\\
	f_{3,m,4,4}\vert S_{2,1}(\tau)& = (-1)^{\frac{m+1}{2}}\left(\frac{3}{64}g_2(\tau)+\frac{1}{64}\Tr(g_3)(\tau)\right).
\end{align*}
Since $ f_{3,m,4,b} $ is a holomorphic cusp form of weight $ 5 $ on $ \Gamma_{64b,4} $, this follows from the Sturm bound as above.

We have $ \alpha_{3,m,4,b}(n) = \frac{1}{3}\gamma_{3,m,4,b}(n) $ for $ n \in \Z_{\geq 1} $, hence
\begin{align*}
	\alpha_{3,m,4,1}(p^r)+2\alpha_{3,m,4,4}(p^r)& = (-1)^{\frac{m+1}{2}}\frac{1}{16}c_{g_2}(p^r),\\
	\alpha_{3,m\bar{p},4,1}(p^{r-2})+2\alpha_{3,m\bar{p},4,4}(p^{r-2})& = (-1)^{\frac{m\bar{p}+1}{2}}\frac{1}{16}c_{g_2}(p^{r-2}).
\end{align*}
The proposition now follows from $ (-1)^{\frac{m+1}{2}} = -\leg{2}{p^r} $ and $ (-1)^{\frac{m\bar{p}+1}{2}} = -\leg{2}{p^r}\leg{-1}{p} $.
\end{proof}

\begin{proposition}
\label{prop:S4}
If $ p \geq 5 $ is prime and $ r \in \Z_{\geq 1} $, then
\[
	S_4^{\Leg}(p^r) = 2p^{3r}-4p^{2r}-9p^r-1+E_4(p^r),
\]
where
\[
	E_4(p^r) = -\frac{1}{2}\left(c_{g_4}(p^r)-p^5c_{g_4}(p^{r-2})\right)-\frac{1}{2}\leg{-1}{p^r}\left(c_{g_5}(p^r)-p^5c_{g_5}(p^{r-2})\right).
\]
\end{proposition}
\begin{proof}
Let $ m \equiv \frac{p^r+1}{2} \pmod{2} $. By \Cref{thm:Seven},
\begin{multline*}
	S_4^{\Leg}(p^r) = C_2\left(p^{3r}-2p^{2r}\right)-3T(4,1)p^r-1+48T(4,0)\left(\alpha_{4,m,2,1}(p^r)-p^5\alpha_{4,m,2,1}(p^{r-2})\right)\\
	+12T(4,1)\left(\alpha_{2,m,2,1}(p^r)-p^3\alpha_{2,m,2,1}(p^{r-2})\right)p^r.
\end{multline*}

We claim that
\begin{align*}
	f_{2,m,4,1}\vert S_{2,1}(\tau)& = 0,\\
	f_{4,m,2,1}\vert S_{2,1}(\tau)& = -\frac{1}{16}\left(g_4(\tau)+(-1)^{m-1}g_5(\tau)\right).
\end{align*}
Since $ f_{2,m,4,1} $, respectively $ f_{4,m,2,1} $, is a holomorphic cusp form of weight $ 4 $, respectively $ 6 $, on $ \Gamma_0(16)\cap \Gamma_1(2) = \Gamma_0(16) $, the claim follows from the Sturm bound and \Cref{lem:S} as in the proof of \Cref{prop:S3}.

Thus,
\begin{align*}
	\alpha_{2,m,2,1}(n)& = \frac{1}{2}\gamma_{2,m,2,1} = 0,\\
	\alpha_{4,m,2,1}(n)& = \frac{1}{6}\gamma_{4,m,2,1}(n) = -\frac{1}{96}\left(c_{g_4}(n)+(-1)^{m-1} c_{g_5}(n)\right)
\end{align*}
for $ n $ odd. Using $ (-1)^{m-1} = \leg{-1}{p^r} $ and substituting these expressions in the formula for $ S_4^{\Leg}(p^r) $ completes the proof.
\end{proof}

\end{document}